\documentclass[a4paper,11pt]{article}
\usepackage{}
\usepackage{mathrsfs}
\usepackage{amsfonts}
\usepackage{amsfonts}
\usepackage{amssymb}
\usepackage{mathrsfs}
\usepackage{indentfirst,latexsym,bm,amsmath,amssymb,amsthm}
\usepackage{xcolor}
\usepackage[
                   bookmarksnumbered=true,
                   bookmarksopen=true, 
                   pdfauthor=kmc,
                   pdfcreator={LaTex with hyperref package + WinEdt},
                   pdftitle=trudge,
                   colorlinks,%
                   citecolor=blue,
                   linkcolor=blue,%
                   urlcolor=blue,
                   hyperindex,%
                   plainpages=false,%
                   pdfstartview=FitH,
                   linktocpage=true,
                   dvipdfm%
                   ]{hyperref}
\usepackage{indentfirst,latexsym,bm,amsmath,amssymb,amsthm}
\usepackage[dvips]{graphicx}

\makeatletter\@addtoreset{equation}{section} \makeatother
\newtheorem{thm}{Theorem}[section]
\newtheorem{Lemma}{Lemma}[section]

\newtheorem{rem}{Remark}[section]

 

\title{Global Nonlinear Stability of Geodesic Solutions of Evolutionary Faddeev Model}

\author{ Jianli Liu\thanks{ Department of Mathematics, Shanghai University, Shanghai 200444, PR China.{ E-mail address: jlliu@shu.edu.cn}.     }~~~Dongbing Zha\thanks{Corresponding author. Department of Mathematics and Institute for Nonlinear Sciences, Donghua University, Shanghai 201620, PR China.{ E-mail address: ZhaDongbing@163.com}.     }~~~Yi Zhou\thanks{ School of Mathematical Sciences, Fudan University, Shanghai 200433, PR China.{ E-mail address: yizhou@fudan.edu.cn}.     }}

\begin{document}

\maketitle
\begin{abstract}
In this paper, for evolutionary Faddeev model corresponding to maps from the Minkowski
space $\mathbb{R}^{1+n}$ to the unit sphere $\mathbb{S}^2$, we show the global nonlinear stability
of geodesic solutions, which are a kind of nontrivial and large solutions.
  \\
\emph{keywords}: Faddeev model; Quasilinear wave equations; Global nonlinear stability.\\
\emph{2010 MSC}: 35L05; 35L72.
\end{abstract}

\pagestyle{plain} \pagenumbering{arabic}

\section{ Introduction and main result  }
In quantum field theory, Faddeev model is an important model that describes heavy elementary particles by knotted topological
solitons. It was introduced by Faddeev in \cite{MR1553682,Fadder} and is a generalization of the well-known classical nonlinear $\sigma$ model of Gell-Mann
and L\'{e}vy \cite{MR0140316}, and is also related closely to the celebrated Skyrme model \cite{MR0138394}. \par
Denote an arbitrary point in Minkowski space $\mathbb{R}^{1+n}$ by
$(t,x)=(x^{\alpha}; 0\leq\alpha\leq n)$ and the space-time derivatives of a function by
$D=(\partial_{t},\nabla)=(\partial_{\alpha}; 0\leq\alpha\leq n).$
We raise and lower indices with the Minkowski metric $\eta=(\eta_{\alpha\beta})=\eta^{-1}=(\eta^{\alpha\beta})=$diag$(1,-1,\cdots,-1)$.
For the Faddeev model, the Lagrangian is given by
\begin{align}\label{labguu}
\mathscr{L}({\bf{n}})=\int_{\mathbb{R}^{1+n}}~\frac{1}{2}\partial_{\mu}{\bf{n}}\cdot\partial^{\mu}{\bf{n}}-\frac{1}{4}
\big(\partial_{\mu}{\bf{n}}\wedge \partial_{\nu}{\bf{n}}\big)\cdot\big(\partial^{\mu}{\bf{n}}\wedge \partial^{\nu}{\bf{n}}\big)~dxdt,
\end{align}
where
$v_1\wedge v_2$ denotes the cross product of the vectors $v_1$  and $v_2$ in
$\mathbb{R}^{3}$ and ${\bf{n}}: \mathbb{R}^{1+n} \longrightarrow \mathbb{S}^2$ is a map from the Minkowski space to the unit sphere in $\mathbb{R}^{3}$. The associated Euler-Lagrange equations take the
form
\begin{align}\label{elFadd}
{\bf{n}}\wedge \partial_{\mu}\partial^{\mu}{\bf{n}}+\big(\partial_{\mu}\big[{\bf{n}}\cdot\big(\partial^{\mu}{\bf{n}}\wedge \partial^{\nu}{\bf{n}}\big)\big]\big)\partial_{\nu}{\bf{n}}=0.
\end{align}
See Faddeev \cite{MR1553682,Fadder,MR1989187}
and Lin and Yang \cite{MR2376667} and references therein. \par
The Faddeev model \eqref{elFadd} was introduced to model elementary particles by using
continuously extended, topologically characterized, relativistically invariant, locally
concentrated, soliton-like fields. The model is not only important in the area of
quantum field theory but also provides many interesting and challenging mathematical
problems, see for examples \cite{Cho,MR852091,MR900505,MR2036370,MR1641192,MR1168556,MR1677728}. There have been
a lot of interesting results in recent years in studying mathematical issues of static Faddeev
model. See Lin and Yang \cite{MR2080954, MR2070206, MR2241558, MR2376667, MR2274465} and Faddeev \cite{MR1989187}. However, the original model \eqref{elFadd} is an evolutionary system, which turns out to be unusual
nonlinear wave equations enjoying the null structure and containing semilinear terms, quasilinear terms and unknowns themselves.  Lei, Lin and Zhou \cite{MR2754038}
is the first rigorous mathematical result on the evolutionary Faddeev model. For the evolutionary Faddeev model in $\mathbb{R}^{1+2}$, they gave the global well-posedness of Cauchy problem for smooth, compact supported initial data with small $H^{11}(\mathbb{R}^2)\times H^{10}(\mathbb{R}^2)$ norm. Under the assumption that the system has equivariant form,  Geba, Nakanishi and Zhang \cite{MR3456696} got the sharp global regularity for the (1+2) dimensional Faddeev model with small critical Besov norm. Large data global well-posedness for the (1+2) dimensional equivariant Faddeev model can be found in Creek \cite{MR3251087} and Geba and Grillakis \cite{MR3909983}. We also refer the readers to Geba and Grillakis's recent monograph \cite{MR3585834} and references therein.
\par
As mentioned above, the equation \eqref{elFadd} for the evolution Faddeev model falls into the form of quasilinear wave equations.
For Cauchy problem of quasilinear wave equations, there are many classical results on global well-posedness of small perturbation of constant trivial solutions.
The global well-posedness for 3-D quasilinear wave equations with null structures and small data can be found in pioneering works Christodoulou \cite{MR820070} and Klainerman \cite{MR837683}. In the 2-D case, Alinhac \cite{MR1856402} first got the global existence of classical solutions with small data. As we known, there are few results on the global regularity of large solutions for quasilinear wave equations. But for some important physical models, the stability of some kind of special large solutions can be studied. For example, for timelike extremal surface equations, codimension one stability of the catenoid was studied in Donninger, Krieger, Szeftel
 and Wong \cite{MR3474816}. Liu and Zhou \cite{ZhouLiu} considered the stability of travelling wave solutions when $n=2$,  and Abbrescia and Wong \cite{Wong} treated the $n\geq 3$ case.
Some results on global nonlinear stability of large solutions for 3-D nonlinear wave equations with null conditions can be found in Alinhac \cite{MR2603759} and Yang \cite{MR3366921}.
\par
The main purpose of this paper is to investigate the global nonlinear stability of geodesic solutions of the evolutionary Faddeev model, which are a kind of nontrivial and large solutions.
The stability of such solutions was first considered by Sideris in the context of wave maps on $\mathbb{R}^{1+3}$ \cite{MR973742}.
Firstly, we rewrite the system \eqref{elFadd} in spherical coordinates. Let
\begin{align}\label{npolar}
{\bf{n}}=(\cos \theta\cos \phi, \cos \theta\sin \phi, \sin \theta)^{\mathbb{T}}
\end{align}
be a vector in the unit sphere. Here $\theta: \mathbb{R}^{1+n} \longrightarrow [-\pi,\pi]$ and $\phi: \mathbb{R}^{1+n} \longrightarrow [-\frac{\pi}{2},\frac{\pi}{2}]$
stand for the latitude and longitude, respectively. Substituting \eqref{npolar} into \eqref{labguu}, we have that the Lagrangian \eqref{labguu} equals to
\begin{align}\label{labguu2}
\mathscr{L}(\theta, \phi)
&=\int_{\mathbb{R}^{1+n}}~\frac{1}{2}Q(\theta,\theta)+\frac{1}{2}\cos^2\theta ~Q(\phi,\phi)-\frac{1}{4}\cos^2\theta~ Q_{\mu\nu}(\theta,\phi)Q^{\mu\nu}(\theta,\phi)~dxdt,
\end{align}
where the null forms
\begin{align}\label{nui899}
Q(f,g)=\partial_{\mu}f\partial^{\mu}g
\end{align}
and
\begin{align}\label{null2}
 Q_{\mu\nu}(f,g)=\partial_{\mu}f\partial_{\nu}g-\partial_{\nu}f\partial_{\mu}g.
\end{align}
By \eqref{labguu2} and Hamilton's principle, we can get the Euler-Lagrange equations with the following form
\begin{align}\label{syst1}
\begin{cases}
\Box \theta=F(\theta, D \theta, D\phi, D^2\theta, D^2\phi),\\
\Box \phi=G(\theta, D \theta, D\phi, D^2\theta, D^2\phi),
\end{cases}
\end{align}
where $\Box=\partial_t^2-\Delta$ is the wave operator on $\mathbb{R}^{1+n}$,
\begin{align}
&F(\theta, D \theta, D\phi, D^2\theta, D^2\phi)\nonumber\\
&=-\frac{1}{2}\sin(2\theta)Q(\phi,\phi)-\frac{1}{4}\sin(2\theta)Q_{\mu\nu}(\theta,\phi)Q^{\mu\nu}(\theta,\phi)\nonumber\\
&~~~-\frac{1}{2}\cos^2\theta Q_{\mu\nu}
\big(\phi,Q^{\mu\nu}(\theta,\phi)\big)
\end{align}
and
\begin{align}
&G(\theta, D \theta, D\phi, D^2\theta, D^2\phi)\nonumber\\
&=\sin^2\theta \Box \phi+\sin(2\theta)Q(\theta,\phi)+\frac{1}{2}\cos^2\theta Q_{\mu\nu}
\big(\theta,Q^{\mu\nu}(\theta,\phi)\big).
\end{align}
\par
We note that if $\Theta=\Theta(t,x)$ satisfies the linear wave equation
\begin{align}\label{xianxing}
\partial_t^2\Theta-\Delta \Theta=0,
\end{align}
then $(\theta,\phi)=(\Theta, 0)$ satisfies the system \eqref{syst1}. In this case, ${\bf{n}}=(\cos \Theta, 0, \sin \Theta)^{\mathbb{T}}$ lies in geodesics on $\mathbb{S}^2$ (i.e. big circles). Thus following the definition in Sideris \cite{MR973742}, we call such solution as geodesic solutions. \par
In this paper, we will investigate the global nonlinear stability of such geodesic solutions of Faddeev model, i.e., the solution $(\Theta, 0)$ of system \eqref{syst1} on $\mathbb{R}^{1+n}, n\geq 2$. Here we will only focus on the cases $n=2$ and $n=3$. As we known, the (1+3) dimensional Faddeev model is an important physical model in particle physics. While the (1+2) dimensional case is much more
complicated than the (1+3) dimensional case from the point of mathematical treating. The $n\geq 4$ case can be treated by a way which is the same with the $n=3$ case. We note that Lei, Lin and Zhou's small data global existence result \cite{MR2754038} can be viewed as some kind of stability result for the trivial geodesic solution $(\theta,\phi)=(0,0)$ of \eqref{syst1} on
$\mathbb{R}^{1+2}$.

\par
The remainder of this introduction will be devoted to the description of some notations, which will be used in the sequel, and statements
of global nonlinear stability theorems in $n=3$ and $n=2$ . In Section 2, some necessary tools used to prove global nonlinear stability theorems are introduced.
The proof of
global nonlinear stability theorems in $n=3$ and $n=2$ will be given in Section 3 and Section 4, respectively.\par
\subsection{Notations}
Firstly, we introduce some vector fields as in Klainerman \cite{MR784477}.
Denote the collection of spatial rotations
$
\Omega=(\Omega_{ij}; 1\leq i<j\leq n)$, where $\Omega_{ij}=x_i\partial_j-x_j\partial_i,
$
the scaling operator
$
S=t\partial_t+x_i\partial_i,
$
and the collection of Lorentz boost operators
$L=(L_i:1\leq i\leq n)$,
$
L_i=t\partial_i+x_i\partial_t,~i=1,\cdots,n.
$
Define the vector fields
$
\Gamma=(D,\Omega, S,L)=(\Gamma_1,\dots,\Gamma_N), N=2+2n+\frac{(n-1)n}{2}.
$
For any given multi-index $a=(a_1,\dots,a_{N}),$
we denote
$
\Gamma^{a}=\Gamma_1^{a_1}\cdots\Gamma_{N}^{a_{N}}.
$
It can be verified that (see \cite{MR1047332})
\begin{align}\label{gooddecay345}
|D u|\leq C\langle t-r\rangle^{-1}|\Gamma u|,
\end{align}
where $\langle \cdot\rangle=(1+|\cdot|^2)^{\frac{1}{2}}$.
We will also introduce the good derivatives (see \cite{MR1856402})
\begin{align}\label{googder}
T_{\mu}=\omega_{\mu}\partial_t+\partial_{\mu},
\end{align}
where $\omega_0=-1, \omega_i=x_i/r~(i=1,\cdots,n), r=|x|$. Denote $T=(T_0,T_1,\cdots,T_n)$. Compared with \eqref{gooddecay345}, we have the following decay estimate:
\begin{align}\label{gooddecay}
|T u|\leq C\langle t+r\rangle^{-1}|\Gamma u|.
\end{align}
\par
The energy associated to the linear wave operator is defined as
\begin{align}
E_1(u(t))=\frac{1}{2}\int_{\mathbb{R}^{n}} \big(|\partial_tu(t,x)|^{2}+ |\nabla u(t,x)|^{2}\big)\, dx,
\end{align}
and the corresponding $k$-th order energy is given by
\begin{align}\label{kord}
E_{k}(u(t))=\sum_{|a|\leq k-1} {E}_1(\Gamma^{a}u(t)).
\end{align}
\par
For getting the global stability of geodesic solutions when $n=2$, we will use some space-time weighted energy estimates and pointwise estimates.
 Let $\sigma=t-r$, $q(\sigma)=\arctan\sigma,
 q'(\sigma)=\frac{1}{1+\sigma^2}=\langle t-r\rangle^{-2}$.  Since $q$ is bounded, there exists a constant $c>1$, such that
\begin{align}\label{noting}
c^{-1}\leq e^{-q(\sigma)}\leq c.
\end{align}
 Following Alinhac \cite{MR1856402}, we can introduce the \lq\lq ghost weight energy"
 \begin{align}
\mathcal {E}_1(u(t))=\frac{1}{2}\int_{\mathbb{R}^{n}}e^{-q(\sigma)}\left<t-r\right>^{-2}{|Tu|^2}\, dx
 \end{align}
 and its $k$-th order version
 \begin{align}
\mathcal {E}_k(u(t))= \sum_{|a|\leq k-1}\mathcal {E}_1(\Gamma^{a}u(t)).
 \end{align}
We will also introduce the following weighted~$L^{\infty}$ norm
\begin{align}
\mathcal {X}_0(u(t))=\big\|\langle t+|\cdot|\rangle^{\frac{n-1}{2}}\langle t-|\cdot|\rangle^{\frac{n-1}{2}}u(t,\cdot)\big\|_{L^{\infty}(\mathbb{R}^{n})},
\end{align}
and its $k$-th order version
\begin{align}\label{kord555}
\mathcal{X}_{k}(u(t))=\sum_{|a|\leq k} {\mathcal {X}}_0(\Gamma^{a}u(t)).
\end{align}
\par
For the convenience,
for any integer $k$ and $1\leq p\leq +\infty$, we will use the following notations
\begin{equation}
\|u(t,\cdot)\|_{W^{k,p}(\mathbb{R}^n)}=\sum_{|a|\leq
k}\|\nabla^{a}u(t,\cdot)\|_{L^{p}(\mathbb{R}^n)},
\end{equation}
\begin{equation}
\|u(t,\cdot)\|_{\dot{W}^{k,p}(\mathbb{R}^n)}=\sum_{|a|=
k}\|\nabla^{a}u(t,\cdot)\|_{L^{p}(\mathbb{R}^n)},
\end{equation}
\begin{equation}
|u(t,\cdot)|_{\Gamma,k}=\sum_{|a|\leq
k}|\Gamma^{a}u(t,\cdot)|
\end{equation}
and
\begin{equation}
\|u(t,\cdot)\|_{\Gamma,k,p}=\sum_{|a|\leq
k}\|\Gamma^{a}u(t,\cdot)\|_{L^{p}(\mathbb{R}^n)}.
\end{equation}
\subsection{Main results}
In this subsection, we will give the global stability results of geodesic solutions to Faddeev model in three and two dimensions. \par
Let $\Theta=\Theta(t,x)$ satisfy
\begin{align}\label{xianxingeeeee}
\begin{cases}
\partial_t^2\Theta-\Delta \Theta=0,~(t,x)\in \mathbb{R}^{1+n}, \\
t=0: \Theta=\Theta_0(x) , \partial_t\Theta=\Theta_1(x),
\end{cases}
\end{align}
where the initial data $\Theta_0$ and $\Theta_1$ are smooth and satisfy
\begin{align}\label{hju7899}
\Theta_0(x)=\Theta_1(x)=0,~~|x|\geq 1.
\end{align}
\par

\par
In the following, we will consider the stability of the geodesic solution $(\Theta, 0)$ of system \eqref{syst1}. Let
\begin{align}\label{shjui89}
(\theta,\phi)=(u+\Theta, v).
\end{align}
We can easily get the equation of $(u,v)$ as following
\begin{align}\label{syst3}
\begin{cases}
\Box u=F(u+\Theta, D(u+\Theta), Dv, D^2(u+\Theta), D^2v),\\
\Box v=G(u+\Theta, D(u+\Theta), Dv, D^2(u+\Theta), D^2v).
\end{cases}
\end{align}
It is obvious that the stability of the geodesics solution $(\Theta, 0)$ of system \eqref{syst1} is equivalent to the stability of zero solution of \eqref{syst3}. Thus we will consider the Cauchy problem of the perturbed system \eqref{syst3}
with initial data
\begin{align}\label{xuyaoop}
t=0: u= u_0(x), \partial_tu= u_1(x),~ v= v_0(x), \partial_tv= v_1(x).
\end{align}
 \par
For introducing the geodesic solution, we note that there are some linear terms in the equation of $v$ in system \eqref{syst3}. Thus in order to ensure the hyperbolicity, we should give some further assumptions on the initial data $(\Theta_0,\Theta_1)$ of system \eqref{xianxingeeeee}.
 When $n=3$, we further assume that
\begin{align}\label{2mmn}
\lambda_0&=\|\Theta_0\|_{\dot{W}^{3,1}(\mathbb{R}^3)}+\|\Theta_1\|_{\dot{W}^{2,1}(\mathbb{R}^3)}<4\pi^2,\\\label{x867uii}
\lambda_1&=\|\Theta_0\|_{\dot{W}^{4,1}(\mathbb{R}^3)}+\|\Theta_1\|_{\dot{W}^{3,1}(\mathbb{R}^3)}<8\pi,\\\label{hu788}
\lambda&=\|\Theta_0\|_{H^{8}(\mathbb{R}^3)}+\|\Theta_1\|_{H^{7}(\mathbb{R}^3)}<+\infty.
\end{align}

 Having set down the necessary notation and formulated Cauchy problem of perturbed system, we are now ready to record our first main result to be proved.
The first main result in this paper is the following
\begin{thm}\label{mainthm}
When $n=3$, assume that $\Theta_0$ and $\Theta_1$ satisfy \eqref{hju7899}, \eqref{2mmn}--\eqref{hu788}, $\Theta$ satisfies \eqref{xianxingeeeee} and $u_0, u_1, v_0$ and $v_1$ are smooth and supported in $|x|\leq 1$.
Then there exist positive constants $A$ and $\varepsilon_0$ such that for any $ 0<\varepsilon\leq\varepsilon_0,$ if
\begin{align}
\|u_0\|_{H^{7}(\mathbb{R}^3)}+\|u_1\|_{H^{6}(\mathbb{R}^3)}+\|v_0\|_{H^{7}(\mathbb{R}^3)}+\|v_1\|_{H^{6}(\mathbb{R}^3)}\leq \varepsilon,
\end{align}
then Cauchy problem \eqref{syst3}--\eqref{xuyaoop} admits a unique global classical solution $(u,v)$ satisfying
\begin{align}\label{labejk99}
\sup_{0\leq t\leq T}\big(E_{7}^{\frac{1}{2}}(u(t))+E_{7}^{\frac{1}{2}}(v(t))\big)\leq A\varepsilon
\end{align}
for any $T>0$.
\end{thm}
When $n=2$, we will assume that
\begin{align}\label{3mmn}
\widetilde{\lambda}_0&=\|\Theta_0\|_{\dot{W}^{2,1}(\mathbb{R}^2)}+\|\Theta_1\|_{\dot{W}^{1,1}(\mathbb{R}^2)}<2\pi,\\\label{xui89755hjj}
\widetilde{\lambda}_1&=\|\Theta_0\|_{\dot{W}^{3,1}(\mathbb{R}^2)}+\|\Theta_1\|_{\dot{W}^{2,1}(\mathbb{R}^2)}<4,\\\label{67yu969}
\widetilde{\lambda}&=\|\Theta_0\|_{W^{10,1}(\mathbb{R}^2)}+\|\Theta_1\|_{W^{9,1}(\mathbb{R}^2)}<+\infty.
\end{align}
\par
The second main result in this paper is the following
\begin{thm}\label{mainthm2}
When $n=2$, assume that $\Theta_0$ and $\Theta_1$ satisfy \eqref{hju7899}, \eqref{3mmn}--\eqref{67yu969}, $\Theta$ satisfies \eqref{xianxingeeeee} and $u_0, u_1, v_0$ and $v_1$ are smooth and supported in $|x|\leq 1$.
Then there exist positive constants $A_1, A_2$ and $\varepsilon_0$ such that for any $ 0<\varepsilon\leq\varepsilon_0,$ if
\begin{align}
\|u_0\|_{H^{7}(\mathbb{R}^2)}+\|u_1\|_{H^{6}(\mathbb{R}^2)}+\|v_0\|_{H^{7}(\mathbb{R}^2)}+\|v_1\|_{H^{6}(\mathbb{R}^2)}\leq \varepsilon,
\end{align}
then Cauchy problem \eqref{syst3}--\eqref{xuyaoop} admits a unique global classical solution $(u,v)$ satisfying
\begin{align}\label{labejk99}
\sup_{0\leq t\leq T}\big(E_{7}^{\frac{1}{2}}(u(t))+E_{7}^{\frac{1}{2}}(v(t))\big)\leq A_1\varepsilon~~\text{and}
\sup_{0\leq t\leq T}\big(\mathcal {X}_{4}(u(t))+\mathcal {X}_{4}(v(t))\big)\leq A_2\varepsilon
\end{align}
for any $T>0$.
\end{thm}

\section{Preliminaries}
\subsection{Commutation relations}
The following lemma concerning the commutation relation between general derivatives, the wave operator and the vector fields was first established
by Klainerman \cite{MR784477}.
\begin{Lemma}\label{LEM2134}
For any given multi-index
$a = (a_1, \dots , a_N)$, we have
\begin{align}\label{shiyi}
[D,\Gamma^{a}]u=\sum_{|b|\leq |a|-1}c_{ab}D\Gamma^{b}u,\\
[\Box,\Gamma^{a}]u=\sum_{|b|\leq |a|-1}C_{ab}\Gamma^{b}\Box u,
\end{align}
where $[\cdot,\cdot]$ stands for the Poisson's bracket, i.e.,~$[A,B]=AB-BA,$ and $c_{ab}$ and $C_{ab}$ are constants.
\end{Lemma}
The following relationship between the vector field $\Gamma$ and null forms can be found in Klainerman \cite{MR837683} .
\begin{Lemma}\label{lem2345}
For null forms $Q(u,v)$ and $Q_{\mu\nu}(u,v)$, we have
 \begin{align}
 \Gamma Q(u,v)=Q(\Gamma u,v)+Q(u,\Gamma v)+\widetilde{Q}(u,v),\\
 \Gamma Q_{\mu\nu}(u,v)=Q_{\mu\nu}(\Gamma u,v)+Q_{\mu\nu}(u,\Gamma v)+\widetilde{Q}_{\mu\nu}(u,v),
 \end{align}
 where $\widetilde{Q}(u,v)$ and $\widetilde{Q}_{\mu\nu}(u,v)$  are some linear combinations of null forms $Q(u,v)$ and $Q_{\mu\nu}(u,v)$.
\end{Lemma}
\subsection{Null form estimates}
The following lemma gives some good decay property concerning the wave operator.
\begin{Lemma}\label{uu679yui}
We have
\begin{align}\label{gjyuii}
(1+ t) |\Box u|\leq C\sum_{|b|\leq 1}|D\Gamma^{b}u|.
\end{align}
\end{Lemma}
\begin{proof}
First, we have the equality
\begin{align}\label{hj7899999}
t\Box u=(t\partial_t+x_i\partial_i)\partial_tu-(x_i\partial_t+t\partial_i)\partial_iu=S\partial_tu-L_i\partial_iu.
\end{align}
Then \eqref{gjyuii} follows from \eqref{hj7899999} and \eqref{shiyi}.
\end{proof}
\begin{Lemma}\label{QLyyy}
For null forms $Q(u,v)$ and $Q_{\mu\nu}(u,v)$, we have
\begin{align}\label{xkktyyyhu}
|Q(u,v)|+|Q_{\mu\nu}(u,v)|\leq C|Du||Tv|+C|Tu||Dv|.
\end{align}
\end{Lemma}
\begin{proof}
By definitions of the null forms \eqref{nui899} and \eqref{null2},  and the good derivatives \eqref{googder}, we have pointwise equalities
\begin{align}\label{fr566999966}
Q(u,v)=T_{\mu}u\partial^{\mu}v-\omega_{\mu}\partial_tuT^{\mu}v
\end{align}
and
\begin{align}\label{fr56666}
Q_{\mu\nu}(u,v)=T_{\mu}u\partial_{\nu}v-T_{\nu}u\partial_{\mu}v-\omega_{\mu}\partial_tuT_{\nu}v+\omega_{\nu}\partial_tuT_{\mu}v.
\end{align}
\eqref{xkktyyyhu} is just a direct consequence of \eqref{fr566999966} and \eqref{fr56666}.
\end{proof}
\begin{Lemma}\label{QL}
For null forms $Q(u,v)$ and $Q_{\mu\nu}(u,v)$, we have
\begin{align}\label{xuyaotyyyhu}
&|\Gamma^aQ(u,v)|+|\Gamma^aQ_{\mu\nu}(u,v)|\leq C\sum_{|b|+|c|\leq |a|}\big(|D\Gamma^{b}u||T\Gamma^{c}v|+|T\Gamma^{b}u||D\Gamma^{c}v|\big)
\end{align}
and
\begin{align}\label{xuyaohu}
&|\Gamma^aQ(u,v)|+|\Gamma^aQ_{\mu\nu}(u,v)|
\leq C\langle t\rangle^{-1}\sum_{|b|+|c|\leq |a|}\big(|D\Gamma^{b}u||\Gamma^{c+1}v|+|\Gamma^{b+1}u||D\Gamma^{c}v|\big).
\end{align}
\begin{proof}
\eqref{xuyaotyyyhu} is a consequence of Lemma \ref{lem2345} and Lemma \ref{QLyyy}.
While \eqref{xuyaohu} follows from \eqref{xuyaotyyyhu} and \eqref{gooddecay}.
\end{proof}
\end{Lemma}

\subsection{Sobolev and Hardy type inequalities}
For getting the decay of derivatives of solutions, we will introduce the following famous Klainerman-Sobolev inequality, which is first proved in Klainerman \cite{MR865359}.
\begin{Lemma}\label{dfft6889}
If $u=u(t,x)$ is a smooth function with sufficient decay at infinity, then we have
\begin{align}\label{Sobo}
\langle t+r\rangle^{\frac{n-1}{2}}\langle t-r\rangle^{\frac{1}{2}}|u(t,x)|\leq C\|u(t,\cdot)\|_{\Gamma,k,2},~k>\frac{n}{2}.
\end{align}
\end{Lemma}
When $n=3$, we can also find the following decay estimates in Klainerman \cite{MR837683}.
\begin{Lemma}\label{rt6788}
If $u=u(t,x)$ is a smooth function with sufficient decay at infinity, then we have
\begin{align}\label{Sobo2}
r^{\frac{1}{2}}|u(t,x)|\leq C\sum_{|a|\leq 1}\|\nabla \Omega^{a}u\|_{L^2(\mathbb{R}^3)}
\end{align}
and
\begin{align}\label{Sobo23333}
r|u(t,x)|\leq C\sum_{|a|\leq 1}\|\nabla \Omega^{a}u\|_{L^2(\mathbb{R}^3)}+C\sum_{|a|\leq 2}\|\Omega^{a}u\|_{L^2(\mathbb{R}^3)}.
\end{align}
\end{Lemma}

The following Hardy type inequality, which is used to produce a general derivative, was first proved in Lindblad \cite{MR1047332}.
\begin{Lemma}\label{rty7ffff78}
 If $u=u(t,x)$ is a smooth function supported in $|x|\leq t+1$, then we have the following Hardy type inequality:
\begin{align}\label{Hardy}
\|\langle t-r\rangle^{-1}u\|_{L^2(\mathbb{R}^n)}\leq C\|\nabla u\|_{L^2(\mathbb{R}^n)}.
\end{align}
\end{Lemma}

\subsection{Estimates of solutions to linear wave equations}
The fundamental theorem of calculus implies the following
\begin{Lemma}\label{xuyao8900}
Let $f:\mathbb{R}^{+}\longrightarrow \mathbb{R}$ be a smooth function with sufficient decay at infinity. Then for any positive integer $m$, we have
\begin{align}
f(t)=\frac{(-1)^{m}}{(m-1)!}\int_{t}^{+\infty}(s-t)^{m-1}f^{(m)}(s)ds.
\end{align}
\end{Lemma}
For getting the stability of geodesic solutions of Faddeev model, we will give some exact boundedness estimates for solutions to homogeneous linear wave equations in two and three dimensions.
\begin{Lemma}\label{3Dbound}
Let $u$ is the solution of the following three dimensional linear wave equation
\begin{align}
\begin{cases}
\Box u(t,x)=0,~(t,x)\in \mathbb{R}^{1+3},\\
t=0: u=u_0(x),\partial_tu=u_1(x),~x\in \mathbb{R}^3,
\end{cases}
\end{align}
where $u_0$ and $u_1$ are smooth functions with compact supports in $|x|\leq 1$.
Then we have
\begin{align}\label{xuyao890}
\|u(t,\cdot)\|_{L^{\infty}(\mathbb{R}^3)}\leq \frac{1}{8\pi}\big(\|u_0\|_{\dot{W}^{3,1}(\mathbb{R}^3)}+\|u_1\|_{\dot{W}^{2,1}(\mathbb{R}^3)}\big)
\end{align}
and
\begin{align}\label{xuyaodee890}
\|\partial_tu(t,\cdot)\|_{L^{\infty}(\mathbb{R}^3)}\leq \frac{1}{8\pi}\big(\|u_0\|_{\dot{W}^{4,1}(\mathbb{R}^3)}+\|u_1\|_{\dot{W}^{3,1}(\mathbb{R}^3)}\big).
\end{align}
\end{Lemma}
\begin{proof}
 By Poisson's formula of three dimensional linear wave equation, we have
\begin{align}\label{shou89}
&u(t,x)\nonumber\\
&=\frac{1}{4\pi t}\int_{|y-x|=t}u_1(y)dS_y+\partial_t\big(\frac{1}{4\pi t}\int_{|y-x|=t}u_0(y)dS_y\big)\nonumber\\
&=\frac{t}{4\pi }\int_{|\omega|=1}u_1(x+t\omega)d\omega+\partial_t\big(\frac{t}{4\pi }\int_{|\omega|=1}u_0(x+t\omega)d\omega\big)\nonumber\\
&=\frac{t}{4\pi }\int_{|\omega|=1}u_1(x+t\omega)d\omega+\frac{t}{4\pi }\int_{|\omega|=1}\partial_t\big(u_0(x+t\omega)\big)d\omega+\frac{1}{4\pi }\int_{|\omega|=1}u_0(x+t\omega)d\omega.
\end{align}
 Lemma \ref{xuyao8900} implies
\begin{align}\label{diyi111}
u_1(x+t\omega)&=\int_{t}^{+\infty}(r-t)\partial^2_r\big(u_1(x+r\omega)\big)dr,\\\label{xjk9089}
\partial_t\big(u_0(x+t\omega)\big)&=\int_{t}^{+\infty}(r-t)\partial^3_r\big(u_0(x+r\omega)\big)dr,\\\label{xjk90l987889}
u_0(x+t\omega)&=-\frac{1}{2}\int_{t}^{+\infty}(r-t)^2\partial^3_r\big(u_0(x+r\omega)\big)dr.
\end{align}
By \eqref{diyi111}, we have
\begin{align}\label{sho9000}
&\Big|\frac{t}{4\pi }\int_{|\omega|=1}u_1(x+t\omega)d\omega\Big|\nonumber\\
&\leq \frac{1}{4\pi }\int_{|\omega|=1}\int_{t}^{+\infty}t(r-t)\big|\partial^2_r\big(u_1(x+r\omega)\big)\big|drd\omega\nonumber\\
&\leq \frac{1}{16\pi }\int_{|\omega|=1}\int_{t}^{+\infty}\big|\partial^2_r\big(u_1(x+r\omega)\big)\big|r^2drd\omega\nonumber\\
&\leq \frac{1}{16\pi}\|u_1\|_{\dot{W}^{2,1}(\mathbb{R}^3)}.
\end{align}
Thanks to \eqref{xjk9089} and \eqref{xjk90l987889}, we also have
\begin{align}\label{s890huy}
&\Big|\frac{t}{4\pi }\int_{|\omega|=1}\partial_t\big(u_0(x+t\omega)\big)d\omega\Big|+\Big|\frac{1}{4\pi }\int_{|\omega|=1}u_0(x+t\omega)d\omega\Big|\nonumber\\
&\leq \frac{1}{8\pi }\int_{|\omega|=1}\int_{t}^{+\infty}(r+t)(r-t)\big|\partial^3_r\big(u_0(x+r\omega)\big)\big|drd\omega\nonumber\\
&\leq \frac{1}{8\pi }\int_{|\omega|=1}\int_{t}^{+\infty}\big|\partial^3_r\big(u_0(x+r\omega)\big)\big|r^2drd\omega\nonumber\\
&\leq \frac{1}{8\pi}\|u_0\|_{\dot{W}^{3,1}(\mathbb{R}^3)}.
\end{align}
Thus, the estimate \eqref{xuyao890} follows from \eqref{shou89}, \eqref{sho9000} and \eqref{s890huy}. Note that $\partial_tu$ satisfies
\begin{align}
\begin{cases}
\Box \partial_tu(t,x)=0,~(t,x)\in \mathbb{R}^{1+3},\\
t=0: \partial_tu=u_1,\partial^2_tu=\Delta u_0,~x\in \mathbb{R}^3.
\end{cases}
\end{align}
Therefore, we can get estimate \eqref{xuyaodee890} similarly.
\end{proof}
\begin{rem}\label{rem45888}
Note that the function $\Theta(t,x)$ satisfies Cauchy problem \eqref{xianxingeeeee}. It follows from Lemma {\rm{\ref{3Dbound}}}, \eqref{2mmn} and \eqref{x867uii} that
\begin{align}\label{xuyao86667790}
|\Theta(t,x)|\leq \frac{1}{8\pi}\big(\|\Theta_0\|_{\dot{W}^{3,1}(\mathbb{R}^3)}+\|\Theta_1\|_{\dot{W}^{2,1}(\mathbb{R}^3)}\big)\leq \frac{1}{8\pi}\lambda_0<\frac{\pi}{2}
\end{align}
and
\begin{align}\label{xuyao666776dee890}
|\partial_t\Theta(t,x)|\leq \frac{1}{8\pi}\big(\|\Theta_0\|_{\dot{W}^{4,1}(\mathbb{R}^3)}+\|\Theta_1\|_{\dot{W}^{3,1}(\mathbb{R}^3)}\big)\leq \frac{1}{8\pi}\lambda_1<1.
\end{align}
\end{rem}

We can also get the following pointwise estimate of linear wave equations in two dimensions.

\begin{Lemma}\label{2Dbound}
Let $u$ is the solution of the following two dimensional linear wave equation
\begin{align}
\begin{cases}
\Box u(t,x)=0,~(t,x)\in \mathbb{R}^{1+2},\\
t=0: u=u_0(x),\partial_tu=u_1(x),~x\in \mathbb{R}^2,
\end{cases}
\end{align}
where $u_0$ and $u_1$ are smooth functions with compact supports in $|x|\leq 1$.
Then we have
\begin{align}\label{hj89900}
\|u(t,\cdot)\|_{L^{\infty}(\mathbb{R}^2)}\leq \frac{1}{4}\big(\|u_0\|_{\dot{W}^{2,1}(\mathbb{R}^2)}+\|u_1\|_{\dot{W}^{1,1}(\mathbb{R}^2)}\big)
\end{align}
and
\begin{align}\label{hjudddddi8900}
\|\partial_tu(t,\cdot)\|_{L^{\infty}(\mathbb{R}^2)}\leq \frac{1}{4}\big(\|u_0\|_{\dot{W}^{3,1}(\mathbb{R}^2)}+\|u_1\|_{\dot{W}^{2,1}(\mathbb{R}^2)}\big).
\end{align}
\end{Lemma}
\begin{proof}
 By Poisson's formula of 2-D linear wave equation, we have
\begin{align}\label{xu89uioo}
&u(t,x)\nonumber\\
&=\frac{1}{2\pi }\int_{|y-x|\leq t}\frac{u_1(y)}{\sqrt{t^2-|y-x|^2}}dy+\partial_t\big(\frac{1}{2\pi }\int_{|y-x|\leq t}\frac{u_0(y)}{\sqrt{t^2-|y-x|^2}}dy\big)\nonumber\\
&=\frac{1}{2\pi }\int_0^{t}\int_{|\omega|=1}\frac{u_1(x+r\omega)}{\sqrt{t^2-r^2}}rd\omega dr+\frac{1}{2\pi t}\int_0^{t}\int_{|\omega|=1}\frac{\partial_r\big(u_0(x+r\omega)\big)}{\sqrt{t^2-r^2}}r^2d\omega dr\nonumber\\
&+\frac{1}{2\pi t}\int_0^{t}\int_{|\omega|=1}\frac{u_0(x+r\omega)}{\sqrt{t^2-r^2}}rd\omega dr.
\end{align}
By Lemma \ref{xuyao8900}, we get
\begin{align}\label{impi98756}
u_1(x+r\omega)&=-\int_{r}^{+\infty}\partial_{\rho}\big(u_1(x+\rho\omega)\big)d\rho,\\\label{xu89dddd876}
\partial_r\big(u_0(x+r\omega)\big)&=-\int_{r}^{+\infty}\partial^2_{\rho}\big(u_0(x+\rho\omega)\big)d\rho,\\\label{impddddddi98756}
u_0(x+r\omega)&=\int_{r}^{+\infty}(\rho-r)\partial^2_{\rho}\big(u_0(x+\rho\omega)\big)d\rho.
\end{align}
Then, \eqref{impi98756} implies
\begin{align}\label{xuyao1fff}
&\Big|\frac{1}{2\pi }\int_0^{t}\int_{|\omega|=1}\frac{u_1(x+r\omega)}{\sqrt{t^2-r^2}}rd\omega dr\Big|\nonumber\\
&\leq \frac{1}{2\pi }\int_0^{t}\frac{1}{\sqrt{t^2-r^2}} dr\sup_{0\leq r\leq t} \int_{r}^{+\infty}\int_{|\omega|=1}\big|\partial_{\rho}\big(u_1(x+\rho\omega)\big)\big|r d\omega d\rho\nonumber\\
&\leq \frac{1}{2\pi }\frac{\pi}{2}\sup_{0\leq r\leq t} \int_{r}^{+\infty}\int_{|\omega|=1}\big|\partial_{\rho}\big(u_1(x+\rho\omega)\big)\big|\rho d\omega d\rho\nonumber\\
&\leq \frac{1}{4}\|u_1\|_{\dot{W}^{1,1}(\mathbb{R}^2)}.
\end{align}
The combination of \eqref{xu89dddd876} and \eqref{impddddddi98756} gives
\begin{align}\label{xuyao5681}
&\Big|\frac{1}{2\pi t}\int_0^{t}\int_{|\omega|=1}\frac{\partial_r\big(u_0(x+r\omega)\big)}{\sqrt{t^2-r^2}}r^2d\omega dr\Big|+\Big|
\frac{1}{2\pi t}\int_0^{t}\int_{|\omega|=1}\frac{u_0(x+r\omega)}{\sqrt{t^2-r^2}}rd\omega dr\Big|\nonumber\\
&\leq \frac{1}{2\pi }\int_0^{t}\frac{1}{\sqrt{t^2-r^2}} dr\sup_{0\leq r\leq t} \int_{r}^{+\infty}\int_{|\omega|=1}\big|\partial^2_{\rho}\big(u_0(x+\rho\omega)\big)\big|\frac{r^2+r(\rho-r)}{t} d\omega d\rho\nonumber\\
&\leq \frac{1}{2\pi }\frac{\pi}{2}\sup_{0\leq r\leq t} \int_{r}^{+\infty}\int_{|\omega|=1}\big|\partial^2_{\rho}\big(u_1(x+\rho\omega)\big)\big|\rho d\omega d\rho\nonumber\\
&\leq \frac{1}{4}\|u_0\|_{\dot{W}^{2,1}(\mathbb{R}^2)}.
\end{align}
Thus \eqref{hj89900} follows from \eqref{xu89uioo}, \eqref{xuyao1fff} and \eqref{xuyao5681}. Noting that $\partial_tu$ satisfies
\begin{align}
\begin{cases}
\Box \partial_tu(t,x)=0,~(t,x)\in \mathbb{R}^{1+2},\\
t=0: \partial_tu=u_1,\partial^2_tu=\Delta u_0,~x\in \mathbb{R}^2,
\end{cases}
\end{align}
we can get \eqref{hjudddddi8900} similarly.
\end{proof}
\begin{rem}\label{rem9999888}
Note that the function $\Theta(t,x)$ satisfies Cauchy problem \eqref{xianxingeeeee}. It follows from Lemma {\rm{\ref{2Dbound}}}, \eqref{3mmn} and \eqref{xui89755hjj} that
\begin{align}\label{xuyao86669999}
|\Theta(t,x)|\leq \frac{1}{4}\big(\|\Theta_0\|_{\dot{W}^{2,1}(\mathbb{R}^2)}+\|\Theta_1\|_{\dot{W}^{1,1}(\mathbb{R}^2)}\big)\leq \frac{1}{4}\widetilde{\lambda}_0<\frac{\pi}{2}
\end{align}
and
\begin{align}\label{xuyao998878ee890}
|\partial_t\Theta(t,x)|\leq \frac{1}{4}\big(\|\Theta_0\|_{\dot{W}^{3,1}(\mathbb{R}^2)}+\|\Theta_1\|_{\dot{W}^{2,1}(\mathbb{R}^2)}\big)\leq \frac{1}{4}\widetilde{\lambda}_1<1.
\end{align}
\end{rem}

The following lemma on $L^1$--$L^{\infty}$ estimates can be found in H\"{o}rmander \cite{MR956961} and Klainerman \cite{MR733719}.
\begin{Lemma}\label{Linfty}
Let $u$ satisfy
\begin{align}\label{Linearwave}
\begin{cases}
\Box u(t,x)=F(t,x),~(t,x)\in \mathbb{R}^{1+n},\\
t=0: u=u_0,\partial_tu=u_1,~x\in \mathbb{R}^n,
\end{cases}
\end{align}
where the initial data $u_0$ and $u_1$ are supported in $|x|\leq 1$.
Then we have
\begin{align}
&\|\langle t+|\cdot|\rangle^{\frac{n-1}{2}}\langle t-|\cdot|\rangle^{l}u(t,\cdot)\|_{L^{\infty}(\mathbb{R}^n)}\nonumber\\
&\leq C\big(\|u_0\|_{W^{n,1}(\mathbb{R}^n)}+\|u_1\|_{W^{n-1,1}(\mathbb{R}^n)}\big) + C\int_0^t(1+\tau)^{-\frac{n-1}{2}+l}\|F(\tau,\cdot)\|_{\Gamma,n-1,1}d\tau,
\end{align}
where $0\leq l\leq \frac{n-1}{2}$.
\end{Lemma}
\subsection{Some estimates on product functions and composite functions}
For getting the estimates of nonlinear terms, we will give the following estimates on product functions and composite functions.
\begin{Lemma}\label{daoshu}
Assume that $u$ and $v$ are smooth functions supported in $|x|\leq t+1$. Then we have
\begin{align}\label{xui89njkhhuio}
\|u Dv\|_{L^{\infty}(\mathbb{R}^3)}
\leq C\langle t\rangle^{-\frac{3}{2}}E_{3}^{\frac{1}{2}}(u(t))E_{3}^{\frac{1}{2}}(v(t)).
\end{align}
\end{Lemma}
\begin{proof}
Without loss of generality, we can assume $t\geq 2$. When $t\leq 2$, \eqref{xui89njkhhuio} is just a consequence of the following Sobolev inequality
\begin{align}\label{xuyooo}
\|u\|_{L^{\infty}(\mathbb{R}^3)}\leq C\|\nabla u\|_{L^{2}(\mathbb{R}^3)}+C\|\nabla^2 u\|_{L^{2}(\mathbb{R}^3)}.
\end{align}
It follows from Klainerman-Sobolev inequality \eqref{Sobo} for $n=3$ and \eqref{xuyooo} that
\begin{align}\label{zhuiyuddf}
&\|uDv\|_{L^{\infty}(r\leq  t/4)} \nonumber\\
&\leq C\langle t\rangle^{-\frac{3}{2}} \|u\|_{L^{\infty}(\mathbb{R}^3)}\|\langle t+r\rangle\langle t-r\rangle^{\frac{1}{2}}Dv\|_{L^{\infty}(r\leq  t/4)}
 \nonumber\\
&\leq C\langle t\rangle^{-\frac{3}{2}}E^{\frac{1}{2}}_3(u(t))E^{\frac{1}{2}}_3(v(t)).
\end{align}
By Klainerman-Sobolev inequality \eqref{Sobo} for $n=3$ and \eqref{Sobo2}, we have
\begin{align}\label{zhufffiyuddf}
&\|uDv\|_{L^{\infty}(r\geq  t/4)} \nonumber\\
&\leq C\langle t\rangle^{-\frac{3}{2}} \|r^{1/2}u\|_{L^{\infty}(r\geq  t/4)}\|\langle t+r\rangle Dv\|_{L^{\infty}(\mathbb{R}^3)}
 \nonumber\\
&\leq C\langle t\rangle^{-\frac{3}{2}}E^{\frac{1}{2}}_3(u(t))E^{\frac{1}{2}}_3(v(t)).
\end{align}
Therefor, noting \eqref{zhuiyuddf} and \eqref{zhufffiyuddf}, we can get the estimate \eqref{xui89njkhhuio}.
\end{proof}
\begin{Lemma}\label{benshen}
Assume that $u$ and $v$ are smooth functions supported in $|x|\leq t+1$. Then we have
\begin{align}
\|u \langle t-r\rangle^{-1}v\|_{L^{\infty}(\mathbb{R}^3)}
\leq C\langle t\rangle^{-\frac{3}{2}}E_{3}^{\frac{1}{2}}(u(t))E_{3}^{\frac{1}{2}}(v(t)).
\end{align}
\end{Lemma}
\begin{proof}
Without loss of generality, we can assume $t\geq 2$.
We have
\begin{align}\label{infwww899}
&\|u \langle t-r\rangle^{-1}v\|_{L^{\infty}(r\geq t/4)}\leq C\langle t\rangle^{-\frac{3}{2}}\| r^{\frac{1}{2}}u\|_{L^{\infty}(\mathbb{R}^3)}\|r \langle t-r\rangle^{-1}v\|_{L^{\infty}(\mathbb{R}^3)}.
\end{align}
Thanks to \eqref{Sobo2}, we can get
\begin{align}\label{infwww89999}
\| r^{\frac{1}{2}}u\|_{L^{\infty}(\mathbb{R}^3)}\leq C\sum_{|\alpha|\leq 1}\|\nabla \Omega^{\alpha}u\|_{L^2(\mathbb{R}^3)}\leq CE_{2}^{\frac{1}{2}}(u(t)).
\end{align}
In view of \eqref{Sobo23333} and Hardy inequality \eqref{Hardy} for $n=3$, we obtain
\begin{align}\label{infwww899999}
&\|r \langle t-r\rangle^{-1}v\|_{L^{\infty}(\mathbb{R}^3)}\nonumber\\
&\leq C\sum_{|\alpha|\leq 1}
\|\nabla (\langle t-r\rangle^{-1}\Omega^{\alpha}v)\|_{L^2(\mathbb{R}^3)}+C\sum_{|\alpha|\leq 2}\|\langle t-r\rangle^{-1} \Omega^{\alpha}v\|_{L^2(\mathbb{R}^3)}\nonumber\\
&\leq C\sum_{|\alpha|\leq 1}
\|\nabla \Omega^{\alpha}v\|_{L^2(\mathbb{R}^3)}+C\sum_{|\alpha|\leq 2}\|\langle t-r\rangle^{-1} \Omega^{\alpha}v\|_{L^2(\mathbb{R}^3)}\nonumber\\
&\leq C\sum_{|\alpha|\leq 2}\|\nabla \Omega^{\alpha}v\|_{L^2(\mathbb{R}^3)}\leq CE_{3}^{\frac{1}{2}}(v(t)).
\end{align}
The combination of \eqref{infwww899}, \eqref{infwww89999} and \eqref{infwww899999} gives
\begin{align}
\|u \langle t-r\rangle^{-1}v\|_{L^{\infty}(r\geq t/4)}
\leq C\langle t\rangle^{-\frac{3}{2}}E_{3}^{\frac{1}{2}}(u(t))E_{3}^{\frac{1}{2}}(v(t)).
\end{align}
The remaining task is to prove
\begin{align}\label{scc89jggg}
\|u \langle t-r\rangle^{-1}v\|_{L^{\infty}(r\leq t/4)}
\leq C\langle t\rangle^{-\frac{3}{2}}E_{3}^{\frac{1}{2}}(u(t))E_{3}^{\frac{1}{2}}(v(t)).
\end{align}
Take a smooth function $\chi$ satisfying
\begin{align}
\chi(\rho)=
\begin{cases}
1,~\rho\leq \frac{1}{4},\\
0,~\rho\geq \frac{1}{2}.
\end{cases}
\end{align}
Then by Sobolev inequality \eqref{xuyooo} and Klainerman-Sobolev inequality \eqref{dfft6889} for $n=3$, we have
\begin{align}\label{123456}
&\|u \langle t-r\rangle^{-1}v\|_{L^{\infty}(r\leq t/4)}\nonumber\\
&
\leq C\langle t\rangle^{-\frac{5}{2}}\|u\|_{L^{\infty}(\mathbb{R}^3)}\|\langle t+r\rangle\langle t-r\rangle^{\frac{1}{2}}\chi (r/t)v\|_{L^{\infty}(\mathbb{R}^3)}\nonumber\\
&\leq C\langle t\rangle^{-\frac{5}{2}}E_{2}^{\frac{1}{2}}(u(t))\|\chi (r/t)v\|_{\Gamma,2,2}.
\end{align}
Now we will prove
\begin{align}\label{1fffff6}
\|\chi (r/t)v\|_{\Gamma,2,2}\leq CtE_{3}^{\frac{1}{2}}(v(t)).
\end{align}
Note that
\begin{align}
\partial_{t}\big(\chi (r/t)\big)=-\frac{r}{t^2}\chi' (r/t),~\partial_{i}\big(\chi (r/t)\big)=\frac{\omega_i}{t}\chi' (r/t),
\end{align}
and
\begin{align}
\Omega_{ij}\big(\chi (r/t)\big)=0, ~S\big(\chi (r/t)\big)=0, ~L_i\big(\chi (r/t)\big)=\omega_i(1-\frac{r^2}{t^2}) \chi' (r/t).
\end{align}
We have
\begin{align}\label{hjidddddddi99}
\sum_{|b|=1}\|\Gamma^{b} \big(\chi (r/t)\big)\|_{L^2(\mathbb{R}^3)}\leq C.
\end{align}
Similarly, we also have
\begin{align}\label{hj89njhuo0}
\sum_{|b|=2}\|\Gamma^b \big(\chi (r/t)\big)\|_{L^2(\mathbb{R}^3)}\leq C.
\end{align}
Thus by \eqref{hjidddddddi99}, \eqref{hj89njhuo0} and Hardy inequality \eqref{Hardy} for $n=3$, we have
\begin{align}
&\|\chi (r/t)v\|_{\Gamma,2,2}\leq C\sum_{|b|+|c|\leq 2}\|\Gamma^{b}\big(\chi (r/t)\big)\Gamma^{c}v\|_{L^2(r\leq t/2)}\nonumber\\
&\leq C\sum_{|c|\leq 2}\|\Gamma^{c}v\|_{L^2(r\leq t/2)}\leq Ct\sum_{|c|\leq 2}\|\langle t-r\rangle^{-1}\Gamma^{c}v\|_{L^2(r\leq t/2)}\leq Ct E_{3}^{\frac{1}{2}}(v(t)).
\end{align}
The combination of \eqref{123456} and \eqref{1fffff6} gives \eqref{scc89jggg}.
\end{proof}
\begin{Lemma}\label{Hardynew2}
Assume that $u, v$ and $w$ are smooth functions supported in $|x|\leq t+1$. If the multi-indices $b,c,d$ satisfy $|b|+|c|+|d|\leq 6$,
we have
\begin{align}\label{hj8ddddddddcc899}
\|\Gamma^{b}uD\Gamma^{c}vD\Gamma^{d}w\|_{L^{2}(\mathbb{R}^3)}\leq C\langle t\rangle^{-\frac{3}{2}}E^{\frac{1}{2}}_7(u(t))E^{\frac{1}{2}}_7(v(t))E^{\frac{1}{2}}_7(w(t)).
\end{align}
\end{Lemma}
\begin{proof}
If $|b|+|c|\leq 3$, it follows from Lemma \ref{daoshu} that
\begin{align}
&\|\Gamma^{b}uD\Gamma^{c}vD\Gamma^{d}w\|_{L^{2}(\mathbb{R}^3)}\nonumber\\
&\leq \|\Gamma^{b}uD\Gamma^{c}v\|_{L^{\infty}(\mathbb{R}^3)}\|D\Gamma^{d}w\|_{L^{2}(\mathbb{R}^3)}\nonumber\\
&\leq
\|\Gamma^{b}uD\Gamma^{c}v\|_{L^{\infty}(\mathbb{R}^3)}E^{\frac{1}{2}}_7(w(t))\nonumber\\
&\leq C\langle t\rangle^{-\frac{3}{2}}E^{\frac{1}{2}}_7(u(t))E^{\frac{1}{2}}_7(v(t))E^{\frac{1}{2}}_7(w(t)).
\end{align}
Using some similar procedure, if $|b|+|d|\leq 3$, we can also get \eqref{hj8ddddddddcc899}.
If $|c|+|d|\leq 3$, by Hardy inequality \eqref{Hardy} for $n=3$, \eqref{gooddecay345} and Lemma \ref{daoshu} , we have
\begin{align}
&\|\Gamma^{b}uD\Gamma^{c}vD\Gamma^{d}w\|_{L^{2}(\mathbb{R}^3)} \nonumber\\
&\leq \|\langle t-r\rangle^{-1}\Gamma^{b}u\|_{L^2(\mathbb{R}^3)} \|\langle t-r\rangle D\Gamma^{c}vD\Gamma^{d}w\|_{L^{\infty}(\mathbb{R}^3)} \nonumber\\
& \leq C\| \Gamma^{c+1}vD\Gamma^{d}w\|_{L^{\infty}(\mathbb{R}^3)}E^{\frac{1}{2}}_7(u(t)) \nonumber\\
&\leq C\langle t\rangle^{-\frac{3}{2}}E^{\frac{1}{2}}_7(u(t))E^{\frac{1}{2}}_7(v(t))E^{\frac{1}{2}}_7(w(t)).
\end{align}
\end{proof}

\begin{Lemma}\label{Hardynew2888}
Assume that $u, v$ and $w$ are smooth functions supported in $|x|\leq t+1$. If the multi-indices $b,c,d$ satisfy $|b|+|c|+|d|\leq 6, |d|\leq 5$,
we have
\begin{align}
\|\Gamma^{b}u \Gamma^{c}vD^2\Gamma^{d}w\|_{L^{2}(\mathbb{R}^3)}\leq C\langle t\rangle^{-\frac{3}{2}}E^{\frac{1}{2}}_7(u(t))E^{\frac{1}{2}}_7(v(t))E^{\frac{1}{2}}_7(w(t)).
\end{align}
\end{Lemma}
\begin{proof}
If $|b|+|d|\leq 3$, it follows from Hardy inequality \eqref{Hardy} for $n=3$, \eqref{gooddecay345} and Lemma \ref{daoshu} that
\begin{align}\label{rfg6888}
&\|\Gamma^{b}u \Gamma^{c}vD^2\Gamma^{d}w\|_{L^{2}(\mathbb{R}^3)}\nonumber\\
&\leq \|\Gamma^{b}u \langle t-r\rangle D^2\Gamma^{d}w\|_{L^{\infty}(\mathbb{R}^3)}\|\langle t-r\rangle^{-1}\Gamma^{c}v\|_{L^{2}(\mathbb{R}^3)} \nonumber\\
&\leq \|\Gamma^{b}u  D\Gamma^{d+1}w\|_{L^{2}(\mathbb{R}^3)} E_{7}^{\frac{1}{2}}(v(t))\nonumber\\
& C\langle t\rangle^{-\frac{3}{2}}E^{\frac{1}{2}}_7(u(t))E^{\frac{1}{2}}_7(v(t))E^{\frac{1}{2}}_7(w(t)).
\end{align}
Using similar procedure, we can also treat the case $|c|+|d|\leq 3$. If $|b|+|c|\leq 3$, by \eqref{gooddecay345} and Lemma \ref{benshen}, we have
\begin{align}\label{rfg6889998}
&\|\Gamma^{b}u \Gamma^{c}vD^2\Gamma^{d}w\|_{L^{2}(\mathbb{R}^3)}\nonumber\\
&\leq \|\Gamma^{b}u \langle t-r\rangle^{-1}\Gamma^{c}v\|_{L^{\infty}(\mathbb{R}^3)}\|\langle t-r\rangle D^2\Gamma^{d}w\|_{L^{2}(\mathbb{R}^3)} \nonumber\\
&\leq C\|\Gamma^{b}u \langle t-r\rangle^{-1}\Gamma^{c}v\|_{L^{\infty}(\mathbb{R}^3)} E_{7}^{\frac{1}{2}}(w(t))\nonumber\\
& \leq C\langle t\rangle^{-\frac{3}{2}}E^{\frac{1}{2}}_7(u(t))E^{\frac{1}{2}}_7(v(t))E^{\frac{1}{2}}_7(w(t)).
\end{align}
\end{proof}
We also have the estimate of composite functions in Li and Zhou \cite{MR3729480} as follows.
\begin{Lemma}\label{composite}
Suppose that $H=H(w)$ is a sufficiently smooth function of $w$ with
\begin{align}
H(w)=\mathscr{O}(|w|^{1+\beta}),
\end{align}
where $\beta\geq 0$ is an integer.
For any given multi-index $a$, if a function $w=w(t,x)$ satisfies
\begin{align}\label{fty}
\|w(t,\cdot)\|_{\Gamma,[\frac{|a|}{2}],\infty}\leq \nu_0,
\end{align}
where $\nu_0$ is a positive constant, then we have the following pointwise estimate
\begin{align}\label{point333}
 |\Gamma^{a}H(w(t,x))|\leq C(\nu_0)\sum_{|l_0|+\cdots+|l_{\beta}|\leq |a|}
 \prod_{j=0}^{\beta}|\Gamma^{l_{j}}w(t,x)|.
 \end{align}
 and $C(\nu_0)$ is a positive constant only depending on $\nu_0$.
\end{Lemma}
\section{Proof of Theorem \ref{mainthm}}
In this section, we shall prove Theorem \ref{mainthm}, i.e., the global nonlinear stability theorem of geodesic solutions for evolutionary Faddeev model when $n=3$,  by some bootstrap argument. Assume that $(u,v)$ is a local classical solution to the Cauchy problem \eqref{syst3}--\eqref{xuyaoop} on $[0, T]$. We will prove that there exist positive constants $A$ and $\varepsilon_0$ such that
\begin{align}
\sup_{0\leq t\leq T}\big(E_{7}^{\frac{1}{2}}(u(t))+E_{7}^{\frac{1}{2}}(v(t))\big)\leq A\varepsilon
\end{align}
under the assumption
\begin{align}
\sup_{0\leq t\leq T}\big(E_{7}^{\frac{1}{2}}(u(t))+E_{7}^{\frac{1}{2}}(v(t))\big)\leq 2A\varepsilon,
\end{align}
where $0<\varepsilon\leq \varepsilon_0$.
\subsection{Energy estimates}
First we will give the estimates on energies $E_{7}(u(t))$ and $E_{7}(v(t))$. For this purpose, it is necessary to introduce some notations about the nonlinear terms on the right hand side of \eqref{syst3}, which will be also used when $n=2$. Denote
\begin{align}\label{333333000}
&F(u+\Theta, D(u+\Theta), Dv, D^2(u+\Theta), D^2v)\nonumber\\
&=a_{\mu\nu}(u+\Theta, Dv)\partial_{\mu}\partial_{\nu}(u+\Theta)+b_{\mu\nu}(u+\Theta, D(u+\Theta), Dv)\partial_{\mu}\partial_{\nu}v\nonumber\\
 &~~+F_1(u+\Theta, D(u+\Theta), Dv),
\end{align}
where
\begin{align}\label{PF00000}
&a_{\mu\nu}(u+\Theta, Dv)\partial_{\mu}\partial_{\nu}(u+\Theta)+b_{\mu\nu}(u+\Theta, D(u+\Theta), Dv)\partial_{\mu}\partial_{\nu}v\nonumber\\
&=
-\frac{1}{2}\cos^2(u+\Theta) Q_{\mu\nu}
\big(v,Q^{\mu\nu}(u+\Theta,v)\big)
\end{align}
and
\begin{align}\label{PF11111}
&F_1(u+\Theta, D(u+\Theta), Dv)\nonumber\\
&=-\frac{1}{2}\sin(2(u+\Theta))Q(v,v)-\frac{1}{4}\sin(2(u+\Theta))Q_{\mu\nu}(u+\Theta,v)Q^{\mu\nu}(u+\Theta,v).
\end{align}
We also denote
\begin{align}
&G(u+\Theta, D(u+\Theta), Dv, D^2(u+\Theta), D^2v)\nonumber\\
&=c_{\mu\nu}(u+\Theta, Dv)\partial_{\mu}\partial_{\nu}(u+\Theta)+d_{\mu\nu}(u+\Theta, D(u+\Theta), Dv)\partial_{\mu}\partial_{\nu}v\nonumber\\
 &~~+G_1(u+\Theta, D(u+\Theta), Dv),
\end{align}
where
\begin{align}\label{PF22222}
&c_{\mu\nu}(u+\Theta,D(u+\Theta), Dv)\partial_{\mu}\partial_{\nu}(u+\Theta)+d_{\mu\nu}(u+\Theta, D(u+\Theta), Dv)\partial_{\mu}\partial_{\nu}v\nonumber\\
&=
\sin^2(u+\Theta) \Box v+\frac{1}{2}\cos^2(u+\Theta) Q_{\mu\nu}
\big(u+\Theta,Q^{\mu\nu}(u+\Theta,v)\big)
\end{align}
and
\begin{align}\label{PF1dddd1111}
G_1(u+\Theta, D(u+\Theta), Dv)
=\sin(2(u+\Theta))Q(u+\Theta,v).
\end{align}

\par
 For any multi-index $a, |a|\leq 6$, taking $\Gamma^{a}$ on the equation \eqref{syst3} and noting Lemma \ref{LEM2134}, we have
\begin{align}\label{rule10}
\Box \Gamma^{a}u&=-\frac{1}{2}\cos^2(u+\Theta) Q_{\mu\nu}
\big(v,Q^{\mu\nu}(\Gamma^au+\Gamma^a\Theta,v)\big)\nonumber\\
&~~~-\frac{1}{2}\cos^2(u+\Theta) Q_{\mu\nu}
\big(v,Q^{\mu\nu}(u+\Theta,\Gamma^av)\big)+f_{a}
\end{align}
and
\begin{align}\label{rule20}
\Box \Gamma^{a}v&=\sin^2(u+\Theta) \Box \Gamma^{a}v+\frac{1}{2}\cos^2(u+\Theta) Q_{\mu\nu}
\big(u+\Theta,Q^{\mu\nu}(\Gamma^{a}u+\Gamma^{a}\Theta,v)\big)\nonumber\\
&~~~+\frac{1}{2}\cos^2(u+\Theta) Q_{\mu\nu}
\big(u+\Theta,Q^{\mu\nu}(u+\Theta,\Gamma^{a}v)\big)+g_{a},
\end{align}
where
\begin{align}\label{faaaa}
&f_a=\big[\Gamma^{a}, a_{\mu\nu}(u+\Theta, Dv)\partial_{\mu}\partial_{\nu}\big](u+\Theta)+\big[\Gamma^a,b_{\mu\nu}(u+\Theta, D(u+\Theta), Dv)\partial_{\mu}\partial_{\nu}\big]v\nonumber\\
&~~~~+\Gamma^{a}F_1(u+\Theta, D(u+\Theta), Dv)+[\Box, \Gamma^{a}]u
\end{align}
and
\begin{align}\label{fdddda}
&g_a=\big[\Gamma^{a}, c_{\mu\nu}(u+\Theta,D(u+\Theta), Dv)\partial_{\mu}\partial_{\nu}\big](u+\Theta)+\big[\Gamma^a,d_{\mu\nu}(u+\Theta, D(u+\Theta), Dv)\partial_{\mu}\partial_{\nu}\big]v\nonumber\\
&~~~~+\Gamma^{a}G_1(u+\Theta, D(u+\Theta), Dv)+[\Box, \Gamma^{a}]v.
\end{align}
 \par
By Leibniz's rule, we have
\begin{align}\label{rule30}
 \langle \partial_t \Gamma^{a}u, \Box \Gamma^{a}u\rangle+ \langle \partial_t \Gamma^{a}v, \Box \Gamma^{a}v\rangle=\partial_te_0+\nabla\cdot q_0,
\end{align}
where
\begin{align}
e_0=\frac{1}{2}\big(|D \Gamma^{a}u|^2+|D \Gamma^{a}v|^2\big).
\end{align}
Leibniz's rule also gives
\begin{align}\label{rule40}
&\langle \partial_t \Gamma^{a}u,  -\frac{1}{2}\cos^2(u+\Theta) Q_{\mu\nu}
\big(v,Q^{\mu\nu}(\Gamma^au+\Gamma^a\Theta,v)\big)\rangle\nonumber\\
&+ \langle \partial_t \Gamma^{a}u,  -\frac{1}{2}\cos^2(u+\Theta) Q_{\mu\nu}
\big(v,Q^{\mu\nu}(u+\Theta,\Gamma^av)\big)\rangle\nonumber\\
&+ \langle \partial_t \Gamma^{a}v, \sin^2(u+\Theta) \Box \Gamma^{a}v\rangle\nonumber\\
&+ \langle \partial_t \Gamma^{a}v,  \frac{1}{2}\cos^2(u+\Theta) Q_{\mu\nu}
\big(u+\Theta,Q^{\mu\nu}(\Gamma^au+\Gamma^a\Theta,v)\big)\rangle\nonumber\\
&+ \langle \partial_t \Gamma^{a}v,  \frac{1}{2}\cos^2(u+\Theta) Q_{\mu\nu}
\big(u+\Theta,Q^{\mu\nu}(u+\Theta,\Gamma^av)\big)\rangle\nonumber\\
&=\partial_t\widetilde{e}+\nabla\cdot \widetilde{q}+\widetilde{p},
\end{align}
where
 \begin{align}
\widetilde{ e}=\widetilde{e}_0+e_1
 \end{align}
 with
 \begin{align}
\widetilde{e}_0&=\frac{1}{2} \sin^2(u+\Theta)|D \Gamma^{a}v|^2+
\cos^2(u+\Theta)\partial_t \Gamma^{a}v\partial_{\mu}\Theta Q^{\mu0}(\Theta,\Gamma^av)\nonumber\\
&-\frac{1}{4}\cos^2(u+\Theta)Q_{\mu\nu}(\Theta,\Gamma^{a}v)Q^{\mu\nu}(\Theta,\Gamma^av)
\end{align}
and
\begin{align}\label{89jjyu6532}
e_1&=-\cos^2(u+\Theta)\partial_t \Gamma^{a}u\partial_{\mu}v\big(Q^{\mu0}(\Gamma^au,v)+Q^{\mu0}(u+\Theta,\Gamma^av)\big)\nonumber\\
&+\cos^2(u+\Theta)\partial_t \Gamma^{a}v\partial_{\mu}(u+\Theta)\big(Q^{\mu0}(\Gamma^au,v)+Q^{\mu0}(u,\Gamma^av)\big)\nonumber\\
&+\cos^2(u+\Theta)\partial_t \Gamma^{a}v\partial_{\mu}uQ^{\mu0}(\Theta,\Gamma^av)\nonumber\\
&+\frac{1}{4}\cos^2(u+\Theta)Q_{\mu\nu}(v,\Gamma^au)\big(Q^{\mu\nu}(\Gamma^au,v)+2Q^{\mu\nu}(u+\Theta,\Gamma^av)\big)\nonumber\\
&-\frac{1}{4}\cos^2(u+\Theta)Q_{\mu\nu}(u+2\Theta,\Gamma^{a}v)Q^{\mu\nu}(u,\Gamma^av),
\end{align}
and
\begin{align}\label{with00000}
\widetilde{p}=
&-\frac{1}{2}\sin(2(u+\Theta))\big(\partial_t \Gamma^{a}vQ(u+\Theta,\Gamma^{a}v)+\partial_i \Gamma^{a}vQ_{0i}(u+\Theta,\Gamma^{a}v)\big)\nonumber\\
&+\frac{1}{2}\cos^2(u+\Theta)\partial_t \Gamma^{a}v Q_{\mu\nu}
\big(u+\Theta,Q^{\mu\nu}(\Gamma^a\Theta,v)\big)\nonumber\\
&+\frac{1}{2}\cos^2(u+\Theta)Q_{\mu\nu}(\partial_tu+\partial_t\Theta,\Gamma^{a}v)Q^{\mu\nu}(u+\Theta,\Gamma^av)\nonumber\\
&-\frac{1}{4}\sin(2(u+\Theta))Q_{\mu\nu}(u+\Theta,\Gamma^{a}v)Q^{\mu\nu}(u+\Theta,\Gamma^av)\nonumber\\
&
-\frac{1}{2}\cos^2(u+\Theta)Q_{\mu\nu}(\partial_tv,\Gamma^au)Q^{\mu\nu}(\Gamma^au,v)\nonumber\\
&
-\frac{1}{2}\cos^2(u+\Theta)\partial_t \Gamma^{a}uQ_{\mu\nu}
\big(v,Q^{\mu\nu}(\Gamma^a\Theta,v)\big)\nonumber\\
&-\frac{1}{2}\sin(2(u+\Theta))\partial_t \Gamma^{a}uQ_{\mu\nu}(v,u+\Theta)Q^{\mu\nu}(\Gamma^au,v)\nonumber\\
&+\frac{1}{4}\sin(2(u+\Theta))(\partial_tu+\partial_t\Theta)Q_{\mu\nu}(v,\Gamma^au)Q^{\mu\nu}(\Gamma^au,v)\nonumber\\
&-\frac{1}{2}\cos^2(u+\Theta)Q_{\mu\nu}(\partial_tv, \Gamma^au)Q^{\mu\nu}(u+\Theta,\Gamma^av)\nonumber\\
&-\frac{1}{2}\cos^2(u+\Theta)Q_{\mu\nu}(v, \Gamma^au)Q^{\mu\nu}(\partial_tu+\partial_t\Theta,\Gamma^av)\nonumber\\
&-\frac{1}{2}\sin(2(u+\Theta))\partial_t \Gamma^{a}uQ_{\mu\nu}(v,u+\Theta)Q^{\mu\nu}(u+\Theta,\Gamma^av)\nonumber\\
&+\frac{1}{2}\sin(2(u+\Theta))(\partial_tu+\partial_t\Theta)Q_{\mu\nu}(v,\Gamma^{a}u)Q^{\mu\nu}(u+\Theta,\Gamma^av).
\end{align}
By \eqref{rule10}, \eqref{rule20}, \eqref{rule30}, \eqref{rule40} and the divergence theorem, we can get
\begin{align}\label{rule50}
&\frac{d}{dt}\int_{\mathbb{R}^3}\big(e_0(t,x)-\widetilde{ e}(t,x)\big) dx\nonumber\\
&\leq \int_{\mathbb{R}^3}|\widetilde{p}(t,x)| dx+\int_{\mathbb{R}^3}|\langle \partial_t \Gamma^{a}u, f_a\rangle| dx+\int_{\mathbb{R}^3}|\langle \partial_t \Gamma^{a}v, g_a\rangle| dx.
\end{align}
Noting
\begin{align}
\widetilde{e}_0
&=\frac{1}{2} \sin^2(u+\Theta)|D \Gamma^{a}v|^2\nonumber\\
&+\frac{1}{2}\cos^2(u+\Theta)\big(2\partial_t \Gamma^{a}v\partial_{\mu}\Theta Q^{\mu0}(\Theta,\Gamma^av)-\frac{1}{2}Q_{\mu\nu}(\Theta,\Gamma^{a}v)Q^{\mu\nu}(\Theta,\Gamma^av)\big)\nonumber\\
&=\frac{1}{2} \sin^2(u+\Theta)|D \Gamma^{a}v|^2\nonumber\\
&+\frac{1}{2}\cos^2(u+\Theta)\big(|\partial_t \Theta|^2|\nabla \Gamma^{a}v|^2-|\nabla \Theta|^2|\partial_t \Gamma^{a}v|^2-\frac{1}{2}Q_{ij}(\Theta,\Gamma^{a}v)Q_{ij}(\Theta,\Gamma^{a}v)\big)\nonumber\\
&=\frac{1}{2} \sin^2(u+\Theta)|D \Gamma^{a}v|^2+\frac{1}{2}\cos^2(u+\Theta)\big(Q(\Theta,\Theta)|\nabla \Gamma^{a}v|^2+(\nabla \Theta\cdot \nabla \Gamma^{a}v)^2\big),
\end{align}
we have
\begin{align}\label{89iytr9900}
&e_0-\widetilde{e}_0\nonumber\\
&=\frac{1}{2}|D \Gamma^{a}u|^2+\frac{1}{2}\cos^2(u+\Theta)|\partial_t \Gamma^{a}v|^2\nonumber\\
&+\frac{1}{2}\cos^2(u+\Theta)\big((1-Q(\Theta,\Theta))|\nabla \Gamma^{a}v|^2-(\nabla \Theta\cdot \nabla \Gamma^{a}v)^2\big)\nonumber\\
&\geq \frac{1}{2}|D \Gamma^{a}u|^2+\frac{1}{2}\cos^2(u+\Theta)|\partial_t \Gamma^{a}v|^2+\frac{1}{2}\cos^2(u+\Theta)(1-|\partial_t\Theta|^2)|\nabla \Gamma^{a}v|^2.
\end{align}
In view of \eqref{89iytr9900} and \eqref{89jjyu6532},
it follows from Remark \ref{rem45888} and the smallness of $|u|, |Du|$ and $|Dv|$ that there exists a positive constant $c_1=c_1(\lambda_0,\lambda_1)$ such that
\begin{align}\label{rtt5666666699kkk8}
c_1^{-1}e_0\leq e_0-\widetilde{e}=e_0-\widetilde{e}_0-e_1 \leq c_1e_0.
\end{align}
\par
Now we estimate all the terms on the right hand side of \eqref{rule50}. In view of \eqref{with00000},
we have
\begin{align}\label{bianjie}
\int_{\mathbb{R}^3}|\widetilde{p}(t,x)| dx
&\leq C\|( |u|+|\Theta|)(|Du|+|D\Theta|)|D\Gamma^{a}v|^2 \|_{L^1(\mathbb{R}^3)}\nonumber\\
&+C\|\big( |Du|+|D\Theta|\big)\big( |D^2u|+|D^2\Theta|\big)|D\Gamma^{a}v|^2 \|_{L^1(\mathbb{R}^3)}\nonumber\\
&+C \|( |Du|+|D\Theta|\big)DvD^2\Gamma^{a}\Theta D\Gamma^{a}v \|_{L^1(\mathbb{R}^3)}\nonumber\\
&+C \|( |Du|+|D\Theta|\big)D^2vD\Gamma^{a}\Theta D\Gamma^{a}v  \|_{L^1(\mathbb{R}^3)}\nonumber\\
&+C \| DvDvD^2\Gamma^{a}\Theta D\Gamma^{a}u \|_{L^1(\mathbb{R}^3)}+C \| DvD^2vD\Gamma^{a}\Theta D\Gamma^{a}u \|_{L^1(\mathbb{R}^3)}\nonumber\\
&+C\|Dv(|D\Theta|+|Dv|+|D^2v|)|D\Gamma^{a}u|^2\|_{L^1(\mathbb{R}^3)}\nonumber\\
&+C\|Dv(|D\Theta|+|Du|)D\Gamma^{a}uD\Gamma^{a}v\|_{L^1(\mathbb{R}^3)}\nonumber\\
&+C\|Dv(|D^2\Theta|+|D^2u|)D\Gamma^{a}uD\Gamma^{a}v\|_{L^1(\mathbb{R}^3)}\nonumber\\
&+C\|D^2v(|D\Theta|+|Du|)D\Gamma^{a}uD\Gamma^{a}v\|_{L^1(\mathbb{R}^3)}.
\end{align}
For the terms on the right hand side of \eqref{bianjie}, the first term is most important. By Lemma \ref{Hardynew2}, we have
\begin{align}
&\|( |u|+|\Theta|)(|Du|+|D\Theta|)|D\Gamma^{a}v|^2 \|_{L^1(\mathbb{R}^3)}\nonumber\\
&\leq \|( |u|+|\Theta|)(|Du|+|D\Theta|)D\Gamma^{a}v\|_{L^2(\mathbb{R}^3)}\|D\Gamma^{a}v \|_{L^2(\mathbb{R}^3)}\nonumber\\
&\leq C\langle t\rangle^{-\frac{3}{2}} \big(E_7(\Theta(t))+E_7(u(t)\big)E_7(v(t)).
\end{align}
By Klainerman-Sobolev inequality \eqref{Sobo},
 we can also get that the remaining terms on the right hand side of \eqref{bianjie} can be controlled by
 \begin{align}
 \langle t\rangle^{-\frac{3}{2}} \big(E_8(\Theta(t))+E_7(u(t)+E_7(v(t)\big)\big(E_7(u(t))+E_7(v(t))\big).
 \end{align}
Therefore, \eqref{bianjie} can be estimated as
\begin{align}\label{biaddddhhhhnjie}
\|\widetilde{p}(t,\cdot)\|_{L^1(\mathbb{R}^3)}\leq C \langle t\rangle^{-\frac{3}{2}} \big(E_8(\Theta(t))+E_7(u(t)+E_7(v(t)\big)\big(E_7(u(t))+E_7(v(t))\big).
\end{align}
By the energy estimate of \eqref{xianxingeeeee}, and noting \eqref{hju7899} and \eqref{hu788}, we can get
\begin{align}\label{7890hjy65}
E^{\frac{1}{2}}_8(\Theta(t))\leq C\big(\|\Theta_0\|_{H^{8}(\mathbb{R}^3)}+\|\Theta_1\|_{H^{7}(\mathbb{R}^3)}\big)\leq C\lambda.
\end{align}
\par
In the following, we will estimate $\|\partial_t \Gamma^{a}uf_a\|_{L^1(\mathbb{R}^3)}$ and $\|\partial_t \Gamma^{a}vg_a\|_{L^1(\mathbb{R}^3)}.$ It is obvious that
\begin{align}
&\|\partial_t \Gamma^{a}uf_a\|_{L^1(\mathbb{R}^3)}+\|\partial_t \Gamma^{a}vg_a\|_{L^1(\mathbb{R}^3)}\nonumber\\
&\leq \big({E}_7^{\frac{1}{2}}(u(t))+ {E}_7^{\frac{1}{2}}(v(t))\big)\big(\|f_a\|_{L^2(\mathbb{R}^3)}+\|g_a\|_{L^2(\mathbb{R}^3)}\big)
\end{align}
and
\begin{align}\label{fgtyyyeee00}
&\|f_a\|_{L^2(\mathbb{R}^3)}+\|g_a\|_{L^2(\mathbb{R}^3)}\nonumber\\
&\leq \|\Gamma^{a}F_1(u+\Theta, D(u+\Theta), Dv)\|_{L^2(\mathbb{R}^3)}+\| \Gamma^{a}G_1(u+\Theta, D(u+\Theta), Dv)\|_{L^2(\mathbb{R}^3)}\nonumber\\
&+
\|\big[\Gamma^{a}, a_{\mu\nu}(u+\Theta, Dv)\partial_{\mu}\partial_{\nu}\big](u+\Theta)\|_{L^2(\mathbb{R}^3)}+\|\big[\Gamma^a,b_{\mu\nu}(u+\Theta, D(u+\Theta), Dv)\partial_{\mu}\partial_{\nu}\big]v\|_{L^2(\mathbb{R}^3)}\nonumber\\
&+
\|\big[\Gamma^{a}, c_{\mu\nu}(u+\Theta,D(u+\Theta), Dv)\partial_{\mu}\partial_{\nu}\big](u+\Theta)\|_{L^2(\mathbb{R}^3)}\nonumber\\
&+\|\big[\Gamma^a,d_{\mu\nu}(u+\Theta, D(u+\Theta), Dv)\partial_{\mu}\partial_{\nu}\big]v\|_{L^2(\mathbb{R}^3)}+\|[\Box, \Gamma^{a}]u\|_{L^2(\mathbb{R}^3)}+\|[\Box, \Gamma^{a}]v\|_{L^2(\mathbb{R}^3)}.
\end{align}
In view of \eqref{PF00000}, \eqref{PF11111}, \eqref{PF22222} and \eqref{PF1dddd1111}, for the terms containing on the right hand side of \eqref{fgtyyyeee00}, we will only focus on the estimates of the following ones
\begin{align}\label{fgtyyyeee8899900888}
&\|\sin(2(u+\Theta))Q(v,v)\|_{\Gamma,6,2}+\|\sin(2(u+\Theta))Q(u+\Theta,v)\|_{\Gamma,6,2}\nonumber\\
&~~~~~~~~~~~~~~~~~~~~+\sum_{\substack{|\beta|+|d|\leq 6\\ |d|\leq 5}}\|\Gamma^{\beta}\sin^2(2(u+\Theta))\Box\Gamma^{d}v\|_{L^2(\mathbb{R}^3)}.
\end{align}
 The remaining terms can be treated similarly. \par
 It follows from Lemma \ref{composite} and Lemma \ref{Hardynew2} that
\begin{align}\label{5677700}
&\|\sin(2(u+\Theta))Q(v,v)\|_{\Gamma,6,2}\nonumber\\
&\leq C\sum_{|b|+|\beta|\leq 6}\|\Gamma^{b}\sin(2(u+\Theta))\Gamma^{\beta}Q(v,v)\|_{L^2(\mathbb{R}^3)}\nonumber\\
&\leq C\sum_{|b|+|c|+|d|\leq 6}\|\Gamma^{b}(u+\Theta)D\Gamma^{c}vD\Gamma^{d}v\|_{L^2(\mathbb{R}^3)}\nonumber\\
&\leq C\langle t\rangle^{-\frac{3}{2}}\big(E_{7}^{\frac{1}{2}}(u(t))+E^{\frac{1}{2}}_{7}(\Theta(t))\big)E(v(t)).
\end{align}
Similarly, we also have
\begin{align}\label{56777fff00}
&\|\sin(2(u+\Theta))Q(u+\Theta,v)\|_{\Gamma,6,2}\nonumber\\
&\leq C\sum_{|b|+|\beta|\leq 6}\|\Gamma^{b}\sin(2(u+\Theta))\Gamma^{\beta}Q(u+\Theta,v)\|_{L^2(\mathbb{R}^3)}\nonumber\\
&\leq C\sum_{|b|+|c|+|d|\leq 6}\|\Gamma^{b}(u+\Theta)D\Gamma^{c}(u+\Theta)D\Gamma^{d}v\|_{L^2(\mathbb{R}^3)}\nonumber\\
&\leq C\langle t\rangle^{-\frac{3}{2}}\big(E_{7}(u(t))+E_{7}(\Theta(t))\big)E^{\frac{1}{2}}(v(t)).
\end{align}
By Lemma \ref{composite} and Lemma \ref{Hardynew2888}, we have
\begin{align}
&\sum_{\substack{|\beta|+|d|\leq 6\\ |d|\leq 5}}\|\Gamma^{\beta}\sin^2(2(u+\Theta))\Box\Gamma^{d}v\|_{L^2(\mathbb{R}^3)}\nonumber\\
&\leq C\sum_{\substack{|b|+|c|+|d|\leq 6\\|d|\leq 5}}\|\Gamma^{b}(u+\Theta)\Gamma^{c}(u+\Theta)D^2\Gamma^{d}v\|_{L^2(\mathbb{R}^3)}\nonumber\\
&\leq C\langle t\rangle^{-\frac{3}{2}}\big(E_{7}(u(t))+E_{7}(\Theta(t))\big)E^{\frac{1}{2}}(v(t)).
\end{align}
From the above discussion, we obtain
\begin{align}\label{xuyoa99088}
&\|\partial_t \Gamma^{a}uf_a\|_{L^1(\mathbb{R}^3)}+\|\partial_t \Gamma^{a}vg_a\|_{L^1(\mathbb{R}^3)}\nonumber\\
&\leq \langle t\rangle^{-\frac{3}{2}} \big(E_8(\Theta(t))+E_7(u(t)+E_7(v(t)\big)\big(E_7(u(t))+E_7(v(t))\big).
\end{align}
\par
Thanks to \eqref{rule50}, \eqref{rtt5666666699kkk8}, \eqref{biaddddhhhhnjie}, \eqref{7890hjy65} and \eqref{xuyoa99088}, we get
\begin{align}
&E_7(u(t))+E_7(v(t))\nonumber\\
&\leq C\varepsilon^2+C\int_0^{t}
\langle \tau\rangle^{-\frac{3}{2}} \big(E_7(u(\tau)+E_7(v(\tau)\big)^2d\tau\nonumber\\
&~~~~~~+C\int_0^{t}
\langle \tau\rangle^{-\frac{3}{2}}E_8(\Theta(\tau))\big(E_7(u(\tau))+E_7(v(\tau))\big)dt\nonumber\\
&\leq C\varepsilon^2+16CA^4\varepsilon^4+C\int_0^{t}
\langle \tau\rangle^{-\frac{3}{2}}\big(E_7(u(\tau))+E_7(v(\tau))\big)dt.
\end{align}
By Gronwall's inequality, we have
\begin{align}\label{hjii908ggg8}
{E}^{\frac{1}{2}}_{7}(u(t)) +{E}^{\frac{1}{2}}_{7}(v(t))\leq C_0\varepsilon+4C_0A^2\varepsilon^2.
\end{align}
\subsection{Conclusion of the proof}
Noting \eqref{hjii908ggg8}, we have obtained
\begin{align}
\sup_{0\leq t\leq T}\big({E}^{\frac{1}{2}}_{7}(u(t)) +{E}^{\frac{1}{2}}_{7}(v(t))\big)\leq C_0\varepsilon+4C_0A^2\varepsilon^2.
\end{align}
Assume that
\begin{align}
E^{\frac{1}{2}}_7(u(0))+E^{\frac{1}{2}}_7(v(0))\leq \widetilde{C}_0\varepsilon^2.
\end{align}
Take $A=\max\{4C_0,4 \widetilde{C}_0\}$ and $\varepsilon_0$ sufficiently small such that
\begin{align}
16C_0A\varepsilon\leq 1.
\end{align}
Then for any $0<\varepsilon\leq \varepsilon_0$, we have
\begin{align}
\sup_{0\leq t\leq T}\big({E}^{\frac{1}{2}}_{7}(u(t)) +{E}^{\frac{1}{2}}_{7}(v(t))\big)\leq A\varepsilon,
\end{align}
which completes the proof of Theorem \ref{mainthm}.
\section{Proof of Theorem \ref{mainthm2}}
In this section, we will prove Theorem \ref{mainthm2}, i.e., the global nonlinear stability theorem of geodesic solutions for evolutionary Faddeev model when $n=2$,  by some suitable bootstrap argument.
We note that in the proof of Theorem \ref{mainthm}, i.e., the global nonlinear stability theorem of geodesic solutions for evolutionary Faddeev model when $n=3$, only the energy estimate is used and the null structure of the system \eqref{syst3} is not employed. The $n=2$ case is much more complicated since the slower decay in time. In order to prove Theorem \ref{mainthm2}, we will exploit the null structure of the system \eqref{syst3} in energy estimates by using Alinhac's ghost weight energy method. To get enough decay rate, we will also use H\"{o}rmander's $L^1$--$L^{\infty}$ estimates, in which the null structure will be also employed. The common feature in the using of these estimates is the sufficient utilization of decay in $\langle t-r\rangle$, besides in $\langle t\rangle$. Some similar idea can be also found in Zha \cite{MR3912654}, which is partially inspired by Alinhac \cite{MR1856402} and Katayama \cite{MR1371789}.
\par
Assume that $(u,v)$ is a local classical solution to Cauchy problem \eqref{syst3}--\eqref{xuyaoop} on $[0, T]$. We will prove that there exist positive constants $A_1, A_2$ and $\varepsilon_0$ such that
\begin{align}
\sup_{0\leq t\leq T}\big(E_{7}^{\frac{1}{2}}(u(t))+E_{7}^{\frac{1}{2}}(v(t))\big)\leq A_1\varepsilon~~\text{and}~~\sup_{0\leq t\leq T}\big(\mathcal {X}_{4}(u(t))+\mathcal {X}_{4}(v(t))\big)\leq A_2\varepsilon
\end{align}
under the assumption
\begin{align}
\sup_{0\leq t\leq T}\big(E_{7}^{\frac{1}{2}}(u(t))+E_{7}^{\frac{1}{2}}(v(t))\big)\leq 2A_1\varepsilon~~\text{and}~~\sup_{0\leq t\leq T}\big(\mathcal {X}_{4}(u(t))+\mathcal {X}_{4}(v(t))\big)\leq 2A_2\varepsilon,
\end{align}
where $0<\varepsilon\leq \varepsilon_0$.
\subsection{Energy estimates}
In this subsection, we will first give the estimates on the energies $E_{7}(u(t))$ and $E_{7}(v(t))$. 
Similarly to the 3-D case, thanks to Lemma \ref{LEM2134}, for any multi-index $a, |a|\leq 6$, we have
\begin{align}\label{rule1}
\Box \Gamma^{a}u&=-\frac{1}{2}\cos^2(u+\Theta) Q_{\mu\nu}
\big(v,Q^{\mu\nu}(\Gamma^au+\Gamma^a\Theta,v)\big)\nonumber\\
&~~~-\frac{1}{2}\cos^2(u+\Theta) Q_{\mu\nu}
\big(v,Q^{\mu\nu}(u+\Theta,\Gamma^av)\big)+f_{a}
\end{align}
and
\begin{align}\label{rule2}
\Box \Gamma^{a}v&=\sin^2(u+\Theta) \Box \Gamma^{a}v+\frac{1}{2}\cos^2(u+\Theta) Q_{\mu\nu}
\big(u+\Theta,Q^{\mu\nu}(\Gamma^{a}u+\Gamma^{a}\Theta,v)\big)\nonumber\\
&~~~+\frac{1}{2}\cos^2(u+\Theta) Q_{\mu\nu}
\big(u+\Theta,Q^{\mu\nu}(u+\Theta,\Gamma^{a}v)\big)+g_{a},
\end{align}
where $f_a$ and $g_a$ are defined through \eqref{faaaa} and \eqref{fdddda}.
 \par
By Leibniz's rule, we have
\begin{align}\label{rule3}
 \langle e^{-q(\sigma)}\partial_t \Gamma^{a}u, \Box \Gamma^{a}u\rangle+ \langle e^{-q(\sigma)}\partial_t \Gamma^{a}v, \Box \Gamma^{a}v\rangle=\partial_te_0+\nabla\cdot q_0+p_0,
\end{align}
where
\begin{align}
e_0=\frac{1}{2}e^{-q(\sigma)}\big(|D \Gamma^{a}u|^2+|D \Gamma^{a}v|^2\big)
\end{align}
and
\begin{align}
p_0=\frac{1}{2}e^{-q(\sigma)} q'(\sigma)\big(|T\Gamma^{a}u|^2+|T\Gamma^{a}v|^2\big).
\end{align}
Leibniz's rule also gives
\begin{align}\label{rule4}
&\langle e^{-q(\sigma)}\partial_t \Gamma^{a}u,  -\frac{1}{2}\cos^2(u+\Theta) Q_{\mu\nu}
\big(v,Q^{\mu\nu}(\Gamma^au+\Gamma^a\Theta,v)\big)\rangle\nonumber\\
&+ \langle e^{-q(\sigma)}\partial_t \Gamma^{a}u,  -\frac{1}{2}\cos^2(u+\Theta) Q_{\mu\nu}
\big(v,Q^{\mu\nu}(u+\Theta,\Gamma^av)\big)\rangle\nonumber\\
&+ \langle e^{-q(\sigma)}\partial_t \Gamma^{a}v, \sin^2(u+\Theta) \Box \Gamma^{a}v\rangle\nonumber\\
&+ \langle e^{-q(\sigma)}\partial_t \Gamma^{a}v,  \frac{1}{2}\cos^2(u+\Theta) Q_{\mu\nu}
\big(u+\Theta,Q^{\mu\nu}(\Gamma^au+\Gamma^a\Theta,v)\big)\rangle\nonumber\\
&+ \langle e^{-q(\sigma)}\partial_t \Gamma^{a}v,  \frac{1}{2}\cos^2(u+\Theta) Q_{\mu\nu}
\big(u+\Theta,Q^{\mu\nu}(u+\Theta,\Gamma^av)\big)\rangle\nonumber\\
&=\partial_t\widetilde{e}+\nabla\cdot \widetilde{q}+\widetilde{p},
\end{align}
where
 \begin{align}
\widetilde{ e}=\widetilde{e}_0+e_1
 \end{align}
 with
 \begin{align}
\widetilde{e}_0&=\frac{1}{2} e^{-q(\sigma)}\sin^2(u+\Theta)|D \Gamma^{a}v|^2\nonumber\\
&+
e^{-q(\sigma)}\cos^2(u+\Theta)\partial_t \Gamma^{a}v\partial_{\mu}\Theta Q^{\mu0}(\Theta,\Gamma^av)\nonumber\\
&-\frac{1}{4}e^{-q(\sigma)}\cos^2(u+\Theta)Q_{\mu\nu}(\Theta,\Gamma^{a}v)Q^{\mu\nu}(\Theta,\Gamma^av)
\end{align}
and
\begin{align}\label{e111990hjuu}
e_1&=-e^{-q(\sigma)}\cos^2(u+\Theta)\partial_t \Gamma^{a}u\partial_{\mu}v\big(Q^{\mu0}(\Gamma^au,v)+Q^{\mu0}(u+\Theta,\Gamma^av)\big)\nonumber\\
&+e^{-q(\sigma)}\cos^2(u+\Theta)\partial_t \Gamma^{a}v\partial_{\mu}(u+\Theta)\big(Q^{\mu0}(\Gamma^au,v)+Q^{\mu0}(u,\Gamma^av)\big)\nonumber\\
&+e^{-q(\sigma)}\cos^2(u+\Theta)\partial_t \Gamma^{a}v\partial_{\mu}uQ^{\mu0}(\Theta,\Gamma^av)\nonumber\\
&+\frac{1}{4}e^{-q(\sigma)}\cos^2(u+\Theta)Q_{\mu\nu}(v,\Gamma^au)\big(Q^{\mu\nu}(\Gamma^au,v)+2Q^{\mu\nu}(u+\Theta,\Gamma^av)\big)\nonumber\\
&-\frac{1}{4}e^{-q(\sigma)}\cos^2(u+\Theta)Q_{\mu\nu}(u+2\Theta,\Gamma^{a}v)Q^{\mu\nu}(u,\Gamma^av),
\end{align}
and
\begin{align}\label{with0}
\widetilde{p}
&=\frac{1}{2}\ e^{-q(\sigma)}\sin^2(u+\Theta)q'(\sigma){|T\Gamma^{a}v|^2}\nonumber\\
&-\frac{1}{2}e^{-q(\sigma)}\sin(2(u+\Theta))\big(\partial_t \Gamma^{a}vQ(u+\Theta,\Gamma^{a}v)+\partial_i \Gamma^{a}vQ_{0i}(u+\Theta,\Gamma^{a}v)\big)\nonumber\\
&+\frac{1}{2}e^{-q(\sigma)}\cos^2(u+\Theta)\partial_t \Gamma^{a}v Q_{\mu\nu}
\big(u+\Theta,Q^{\mu\nu}(\Gamma^a\Theta,v)\big)\nonumber\\
&+\frac{1}{2}e^{-q(\sigma)}\cos^2(u+\Theta)Q_{\mu\nu}(\partial_tu+\partial_t\Theta,\Gamma^{a}v)Q^{\mu\nu}(u+\Theta,\Gamma^av)\nonumber\\
&-e^{-q(\sigma)}q'(\sigma)\cos^2(u+\Theta)T_{\nu} \Gamma^{a}v\partial_{\mu}(u+\Theta)Q^{\mu\nu}(u+\Theta,\Gamma^av)\nonumber\\
&+\frac{1}{4}e^{-q(\sigma)}q'(\sigma)\cos^2(u+\Theta)Q_{\mu\nu}(u+\Theta,\Gamma^{a}v)Q^{\mu\nu}(u+\Theta,\Gamma^av)\nonumber\\
&-\frac{1}{4}e^{-q(\sigma)}\sin(2(u+\Theta))Q_{\mu\nu}(u+\Theta,\Gamma^{a}v)Q^{\mu\nu}(u+\Theta,\Gamma^av)\nonumber\\
&
-\frac{1}{2}e^{-q(\sigma)}\cos^2(u+\Theta)Q_{\mu\nu}(\partial_tv,\Gamma^au)Q^{\mu\nu}(\Gamma^au,v)\nonumber\\
&
-\frac{1}{2}e^{-q(\sigma)}\cos^2(u+\Theta)\partial_t \Gamma^{a}uQ_{\mu\nu}
\big(v,Q^{\mu\nu}(\Gamma^a\Theta,v)\big)\nonumber\\
&+e^{-q(\sigma)}q'(\sigma)\cos^2(u+\Theta)T_{\nu} \Gamma^{a}u\partial_{\mu}vQ^{\mu\nu}(\Gamma^au,v)\nonumber\\
&-\frac{1}{4}e^{-q(\sigma)}q'(\sigma)\cos^2(u+\Theta)Q_{\mu\nu}(v,\Gamma^au)Q^{\mu\nu}(\Gamma^au,v)\nonumber\\
&-\frac{1}{2}e^{-q(\sigma)}\sin(2(u+\Theta))\partial_t \Gamma^{a}uQ_{\mu\nu}(v,u+\Theta)Q^{\mu\nu}(\Gamma^au,v)\nonumber\\
&+\frac{1}{4}e^{-q(\sigma)}\sin(2(u+\Theta))(\partial_tu+\partial_t\Theta)Q_{\mu\nu}(v,\Gamma^au)Q^{\mu\nu}(\Gamma^au,v)\nonumber\\
&-\frac{1}{2}e^{-q(\sigma)}\cos^2(u+\Theta)Q_{\mu\nu}(\partial_tv, \Gamma^au)Q^{\mu\nu}(u+\Theta,\Gamma^av)\nonumber\\
&-\frac{1}{2}e^{-q(\sigma)}\cos^2(u+\Theta)Q_{\mu\nu}(v, \Gamma^au)Q^{\mu\nu}(\partial_tu+\partial_t\Theta,\Gamma^av)\nonumber\\
&+e^{-q(\sigma)}q'(\sigma)\cos^2(u+\Theta)T_{\nu} \Gamma^{a}u\partial_{\mu}vQ^{\mu\nu}(u+\Theta,\Gamma^av)
\end{align}
\begin{align}
&-\frac{1}{2}e^{-q(\sigma)}\sin(2(u+\Theta))\partial_t \Gamma^{a}uQ_{\mu\nu}(v,u+\Theta)Q^{\mu\nu}(u+\Theta,\Gamma^av)\nonumber\\
&+\frac{1}{2}e^{-q(\sigma)}\sin(2(u+\Theta))(\partial_tu+\partial_t\Theta)Q_{\mu\nu}(v,\Gamma^{a}u)Q^{\mu\nu}(u+\Theta,\Gamma^av)\nonumber\\
&-e^{-q(\sigma)}q'(\sigma)\cos^2(u+\Theta)T_{\nu} \Gamma^{a}v\partial_{\mu}(u+\Theta)Q^{\mu\nu}(\Gamma^au,v)\nonumber\\
&+\frac{1}{2}e^{-q(\sigma)}q'(\sigma)\cos^2(u+\Theta)Q_{\mu\nu}(u+\Theta,\Gamma^{a}v)Q^{\mu\nu}(\Gamma^au,v).
\end{align}
By \eqref{rule1}, \eqref{rule2}, \eqref{rule3}, \eqref{rule4} and the divergence theorem, we can get
\begin{align}\label{rule5}
&\frac{d}{dt}\int_{\mathbb{R}^2}\big(e_0(t,x)-\widetilde{ e}(t,x)\big) dx+\int_{\mathbb{R}^2}p_0(t,x) dx\nonumber\\
&\leq \int_{\mathbb{R}^2}|\widetilde{p}(t,x)| dx+\int_{\mathbb{R}^2}|\langle e^{-q(\sigma)}\partial_t \Gamma^{a}u, f_a\rangle| dx+\int_{\mathbb{R}^2}|\langle e^{-q(\sigma)}\partial_t \Gamma^{a}v, g_a\rangle| dx.
\end{align}
Noting
\begin{align}
\widetilde{e}_0
&=\frac{1}{2} e^{-q(\sigma)}\sin^2(u+\Theta)|D \Gamma^{a}v|^2\nonumber\\
&+\frac{1}{2}e^{-q(\sigma)}\cos^2(u+\Theta)\big(2\partial_t \Gamma^{a}v\partial_{\mu}\Theta Q^{\mu0}(\Theta,\Gamma^av)-\frac{1}{2}Q_{\mu\nu}(\Theta,\Gamma^{a}v)Q^{\mu\nu}(\Theta,\Gamma^av)\big)\nonumber\\
&=\frac{1}{2} e^{-q(\sigma)}\sin^2(u+\Theta)|D \Gamma^{a}v|^2\nonumber\\
&+\frac{1}{2}e^{-q(\sigma)}\cos^2(u+\Theta)\big(|\partial_t \Theta|^2|\nabla \Gamma^{a}v|^2-|\nabla \Theta|^2|\partial_t \Gamma^{a}v|^2-\frac{1}{2}Q_{ij}(\Theta,\Gamma^{a}v)Q_{ij}(\Theta,\Gamma^{a}v)\big)\nonumber\\
&=\frac{1}{2} e^{-q(\sigma)}\sin^2(u+\Theta)|D \Gamma^{a}v|^2\nonumber\\
&+\frac{1}{2}e^{-q(\sigma)}\cos^2(u+\Theta)\big(Q(\Theta,\Theta)|\nabla \Gamma^{a}v|^2+(\nabla \Theta\cdot \nabla \Gamma^{a}v)^2\big),
\end{align}
we have
\begin{align}\label{hj89eeee99}
&e_0-\widetilde{e}_0\nonumber\\
&=\frac{1}{2}e^{-q(\sigma)}|D \Gamma^{a}u|^2+\frac{1}{2}e^{-q(\sigma)}\cos^2(u+\Theta)|\partial_t \Gamma^{a}v|^2\nonumber\\
&+\frac{1}{2}e^{-q(\sigma)}\cos^2(u+\Theta)\big((1-Q(\Theta,\Theta))|\nabla \Gamma^{a}v|^2-(\nabla \Theta\cdot \nabla \Gamma^{a}v)^2\big)\nonumber\\
&\geq \frac{1}{2}e^{-q(\sigma)}|D \Gamma^{a}u|^2+\frac{1}{2}e^{-q(\sigma)}\cos^2(u+\Theta)|\partial_t \Gamma^{a}v|^2\nonumber\\
&+\frac{1}{2}e^{-q(\sigma)}\cos^2(u+\Theta)(1-|\partial_t\Theta|^2)|\nabla \Gamma^{a}v|^2.
\end{align}
By \eqref{hj89eeee99}, \eqref{e111990hjuu}, Remark \ref{rem9999888}
 and the smallness of $|u|, |Du|$ and $|Dv|$,
we can obtain that there exists a positive constant $c_2=c_2(\widetilde{\lambda}_0,\widetilde{\lambda}_1)$ such that
\begin{align}\label{rtt56666666}
c_2^{-1}e_0\leq e_0-\widetilde{e}=e_0-\widetilde{e}_0-e_1 \leq c_2e_0.
\end{align}
\par
Now we will estimate all the terms on the right hand side of \eqref{rule5}. Thanks to \eqref{with0} and Lemma \ref{QL}, we have the pointwise estimate
\begin{align}
|\widetilde{p}|&\leq C\big(|u|^2_{\Gamma,2}+|v|^2_{\Gamma,2}+|\Theta|^2_{\Gamma,8}\big)\big(|Du|_{\Gamma,6}+|Dv|_{\Gamma,6}\big)(\sum_{||b|\leq 6}|T\Gamma^{b}u|+\sum_{||b|\leq 6}|T\Gamma^{b}v|)\nonumber\\
&+C\langle t\rangle^{-1}\big(|u|^2_{\Gamma,2}+|v|^2_{\Gamma,2}+|\Theta|^2_{\Gamma,8}\big)\big(|Du|_{\Gamma,6}+|Dv|_{\Gamma,6}\big)^2.
\end{align}
Thus we have
\begin{align}\label{y78900}
&\|\widetilde{p}(t,\cdot)\|_{L^1(\mathbb{R}^2)}\nonumber\\
&\leq C\langle t\rangle^{-1}\big(\mathcal {X}^2_2(u(t))+\mathcal {X}^2_2(v(t))+\mathcal {X}^2_8({\Theta}(t))\big)\big(E_7^{\frac{1}{2}}(u(t))+E_7^{\frac{1}{2}}(v(t))\big)\big(\mathcal {E}_7^{\frac{1}{2}}(u(t))+\mathcal {E}_7^{\frac{1}{2}}(v(t))\big) \nonumber\\
&+ C \langle t\rangle^{-2}\big(\mathcal {X}^2_2(u(t))+\mathcal {X}^2_2(v(t))+\mathcal {X}^2_8({\Theta}(t))\big)\big(E_7(u(t))+E_7(v(t))\big).
\end{align}
It follows from \eqref{xianxingeeeee}£¬ \eqref{hju7899}, \eqref{67yu969} and Lemma \ref{Linfty} that
\begin{align}\label{4890uoop}
\mathcal {X}_8({\Theta}(t))\leq C\big(\|\Theta_0\|_{W^{10,1}(\mathbb{R}^2)}+\|\Theta_1\|_{W^{9,1}(\mathbb{R}^2)}\big)\leq C\widetilde{\lambda}.
\end{align}
\par
Now we estimate $\|\partial_t \Gamma^{a}uf_a\|_{L^1(\mathbb{R}^2)}$ and $\|\partial_t \Gamma^{a}vg_a\|_{L^1(\mathbb{R}^2)}.$ It is obvious that
\begin{align}
&\|\partial_t \Gamma^{a}uf_a\|_{L^1(\mathbb{R}^2)}+\|\partial_t \Gamma^{a}vg_a\|_{L^1(\mathbb{R}^2)}\nonumber\\
&\leq \big({E}_7^{\frac{1}{2}}(u(t))+ {E}_7^{\frac{1}{2}}(v(t))\big)\big(\|f_a\|_{L^2(\mathbb{R}^2)}+\|g_a\|_{L^2(\mathbb{R}^2)}\big)
\end{align}
and
\begin{align}\label{fgtyyyeee}
&\|f_a\|_{L^2(\mathbb{R}^2)}+\|g_a\|_{L^2(\mathbb{R}^2)}\nonumber\\
&\leq \|\Gamma^{a}F_1(u+\Theta, D(u+\Theta), Dv)\|_{L^2(\mathbb{R}^2)}+\| \Gamma^{a}G_1(u+\Theta, D(u+\Theta), Dv)\|_{L^2(\mathbb{R}^2)}\nonumber\\
&+
\|\big[\Gamma^{a}, a_{\mu\nu}(u+\Theta, Dv)\partial_{\mu}\partial_{\nu}\big](u+\Theta)\|_{L^2(\mathbb{R}^2)}+\|\big[\Gamma^a,b_{\mu\nu}(u+\Theta, D(u+\Theta), Dv)\partial_{\mu}\partial_{\nu}\big]v\|_{L^2(\mathbb{R}^2)}\nonumber\\
&+
\|\big[\Gamma^{a}, c_{\mu\nu}(u+\Theta,D(u+\Theta), Dv)\partial_{\mu}\partial_{\nu}\big](u+\Theta)\|_{L^2(\mathbb{R}^2)}+\|[\Box, \Gamma^{a}]u\|_{L^2(\mathbb{R}^2)}\nonumber\\
&+\|\big[\Gamma^a,d_{\mu\nu}(u+\Theta, D(u+\Theta), Dv)\partial_{\mu}\partial_{\nu}\big]v\|_{L^2(\mathbb{R}^2)}+\|[\Box, \Gamma^{a}]v\|_{L^2(\mathbb{R}^2)}.
\end{align}
We will only focus on the estimates of the first and second parts on the right hand side of \eqref{fgtyyyeee}, the remaining parts can be treated similarly.
In view of \eqref{PF11111} and \eqref{PF1dddd1111}, we have
\begin{align}\label{fgtyyyeee88999}
& \|\Gamma^{a}F_1(u+\Theta, D(u+\Theta), Dv)_{L^2(\mathbb{R}^2)}+\| \Gamma^{a}G_1(u+\Theta, D(u+\Theta), Dv)\|_{L^2(\mathbb{R}^2)}\nonumber\\
&\leq \|\sin(2(u+\Theta))Q(v,v)\|_{\Gamma,6,2}+\|\sin(2(u+\Theta))Q(u+\Theta,v)\|_{\Gamma,6,2}\nonumber\\
&+\|\sin(2(u+\Theta))Q_{\mu\nu}(u+\Theta,v)Q^{\mu\nu}(u+\Theta,v)\|_{\Gamma,6,2}.
\end{align}
It follows from Lemma \ref{composite} and Lemma \ref{QL} that
\begin{align}\label{56777}
&\|\sin(2(u+\Theta))Q(v,v)\|_{\Gamma,6,2}\nonumber\\
&\leq C\sum_{|b|+|\beta|\leq 6}\|\Gamma^{b}\sin(2(u+\Theta))\Gamma^{\beta}Q(v,v)\|_{L^2(\mathbb{R}^2)}\nonumber\\
&\leq C\sum_{|b|+|c|+|d|\leq 6}\|\Gamma^{b}u D\Gamma^{c}v T\Gamma^{d}v\|_{L^2(\mathbb{R}^2)}+ C\sum_{|b|+|c|+|d|\leq 6}\|\Gamma^{b}\Theta D\Gamma^{c}v T\Gamma^{d}v\|_{L^2(\mathbb{R}^2)}.
\end{align}
For $|b|+|c|+|d|\leq 6$, if $|b|+|c|\leq 3$, we have
\begin{align}\label{fgrtt}
&\|\Gamma^{b}u D\Gamma^{c}v T\Gamma^{d}v\|_{L^2(\mathbb{R}^2)}\nonumber\\
&\leq C\langle t\rangle^{-1} \|\langle t+r\rangle^{\frac{1}{2}}\langle t-r\rangle^{\frac{1}{2}}\Gamma^{b}u\|_{L^{\infty}(\mathbb{R}^2)} \|\langle t+r\rangle^{\frac{1}{2}}\langle t-r\rangle^{\frac{1}{2}} D\Gamma^{c}v\|_{L^{\infty}(\mathbb{R}^2)}  \|\langle t-r\rangle^{-1} T\Gamma^{d}v\|_{L^2(\mathbb{R}^2)}\nonumber\\
&\leq C\langle t\rangle^{-1} \mathcal {X}_{4}(u(t))\mathcal {X}_{4}(v(t))\mathcal {E}^{\frac{1}{2}}_{7}(v(t)).
\end{align}
If $|b|+|d|\leq 3$, by \eqref{gooddecay} we get
\begin{align}\label{xiaopjk}
&\|\Gamma^{b}u D\Gamma^{c}v T\Gamma^{d}v\|_{L^2(\mathbb{R}^2)}\nonumber\\
&\leq C\langle t\rangle^{-2} \|\langle t+r\rangle^{\frac{1}{2}}\langle t-r\rangle^{\frac{1}{2}}\Gamma^{b}u\|_{L^{\infty}(\mathbb{R}^2)} \|D\Gamma^{c}v\|_{L^{2}(\mathbb{R}^2)}  \|\langle t+r\rangle^{\frac{1}{2}}\langle t-r\rangle^{\frac{1}{2}}\Gamma^{d+1}v\|_{L^{\infty}(\mathbb{R}^2)}\nonumber\\
&\leq C\langle t\rangle^{-2} \mathcal {X}_{4}(u(t))\mathcal {X}_{4}(v(t)) {E}^{\frac{1}{2}}_{7}(v(t)).
\end{align}
If $|c|+|d|\leq 3$, by Hardy inequality \eqref{Hardy} and \eqref{gooddecay} we have
\begin{align}
&\|\Gamma^{b}u D\Gamma^{c}v T\Gamma^{d}v\|_{L^2(\mathbb{R}^2)}\nonumber\\
&\leq C\langle t\rangle^{-2} \|\langle t-r\rangle^{-1}\Gamma^{b}u\|_{L^{2}(\mathbb{R}^2)} \|\langle t+r\rangle^{\frac{1}{2}}\langle t-r\rangle^{\frac{1}{2}}D\Gamma^{c}v\|_{L^{{\infty}}(\mathbb{R}^2)}  \|\langle t+r\rangle^{\frac{1}{2}}\langle t-r\rangle^{\frac{1}{2}}\Gamma^{d+1}v\|_{L^{\infty}(\mathbb{R}^2)}\nonumber\\
&\leq C\langle t\rangle^{-2} \mathcal {X}_{4}(v(t))\mathcal {X}_{4}(v(t)) {E}^{\frac{1}{2}}_{7}(u(t)).
\end{align}
Thus we obtain
\begin{align}\label{shj799}
&\sum_{|b|+|c|+|d|\leq 6}\|\Gamma^{b}u D\Gamma^{c}v T\Gamma^{d}v\|_{L^2(\mathbb{R}^2)}\nonumber\\
&\leq C\langle t\rangle^{-1} \big(\mathcal {X}^2_{4}(u(t))+\mathcal {X}^2_{4}(v(t))\big)\big(\mathcal {E}^{\frac{1}{2}}_{7}(u(t))\mathcal +\mathcal{E}^{\frac{1}{2}}_{7}(v(t))\big)\nonumber\\
&+ C\langle t\rangle^{-2}\big(\mathcal {X}^2_{4}(u(t))+\mathcal {X}^2_{4}(v(t))\big)\big( {E}^{\frac{1}{2}}_{7}(u(t)) +{E}^{\frac{1}{2}}_{7}(v(t))\big).
\end{align}
For the second term on the right hand side of \eqref{56777}, if $|b|+|c|\leq 3$, similarly to \eqref{fgrtt}, we have
\begin{align}\label{fgrtt111}
\|\Gamma^{b}\Theta D\Gamma^{c}v T\Gamma^{d}v\|_{L^2(\mathbb{R}^2)}\leq C\langle t\rangle^{-1} \mathcal {X}_{4}(\Theta(t))\mathcal {X}_{4}(v(t))\mathcal {E}^{\frac{1}{2}}_{7}(v(t)).
\end{align}
if $|b|+|d|\leq 3$ or $|c|+|d|\leq 3$, similarly to \eqref{xiaopjk}, we have
\begin{align}\label{xiaopjk1111}
\|\Gamma^{b}\Theta D\Gamma^{c}v T\Gamma^{d}v\|_{L^2(\mathbb{R}^2)}\leq C\langle t\rangle^{-2} \mathcal {X}_{8}(\Phi(t))\mathcal {X}_{4}(v(t)) {E}^{\frac{1}{2}}_{7}(v(t)).
\end{align}
Thus we have
\begin{align}\label{frt566tttt}
&\sum_{|b|+|c|+|d|\leq 6}\|\Gamma^{b}\Theta D\Gamma^{c}v T\Gamma^{d}v\|_{L^2(\mathbb{R}^2)}\nonumber\\
&\leq C\langle t\rangle^{-1} \big(\mathcal {X}^2_{8}(\Phi(t))+\mathcal {X}^2_{4}(v(t))\big)\mathcal{E}^{\frac{1}{2}}_{7}(v(t))
+ C\langle t\rangle^{-2}\big(\mathcal {X}^2_{8}(\Phi(t))+\mathcal {X}^2_{4}(v(t))\big){E}^{\frac{1}{2}}_{7}(v(t)).
\end{align}
By \eqref{56777}, \eqref{shj799} and \eqref{frt566tttt}, we have
\begin{align}\label{56777yyyy}
&\|\sin(2(u+\Theta))Q(v,v)\|_{\Gamma,6,2}\nonumber\\
&\leq C\langle t\rangle^{-1} \big(\mathcal {X}^2_{4}(u(t))+\mathcal {X}^2_{4}(v(t))+\mathcal {X}^2_{8}(\Phi(t))\big)\big(\mathcal {E}^{\frac{1}{2}}_{7}(u(t))\mathcal +\mathcal{E}^{\frac{1}{2}}_{7}(v(t))\big)\nonumber\\
&+ C\langle t\rangle^{-2}\big(\mathcal {X}^2_{4}(u(t))+\mathcal {X}^2_{4}(v(t))+\mathcal {X}^2_{8}(\Phi(t))\big)\big( {E}^{\frac{1}{2}}_{7}(u(t)) +{E}^{\frac{1}{2}}_{7}(v(t))\big).
\end{align}
\par
Similarly to \eqref{56777yyyy}, the second and third part on the right hand side of \eqref{fgtyyyeee88999} can be estimated by the same way and admit the same upper bound.\par
From the above discussion, we can get
\begin{align}\label{ccfr5666699}
&\|\partial_t \Gamma^{a}uf_a\|_{L^1(\mathbb{R}^2)}+\|\partial_t \Gamma^{a}vg_a\|_{L^1(\mathbb{R}^2)}\nonumber\\
&\leq C\langle t\rangle^{-1} \big(\mathcal {X}^2_{4}(u(t))+\mathcal {X}^2_{4}(v(t))+\mathcal {X}^2_{8}(\Phi(t))\big)\big(\mathcal {E}^{\frac{1}{2}}_{7}(u(t)) +\mathcal{E}^{\frac{1}{2}}_{7}(v(t))\big)\big( {E}^{\frac{1}{2}}_{7}(u(t)) +{E}^{\frac{1}{2}}_{7}(v(t))\big)\nonumber\\
&+ C\langle t\rangle^{-2}\big(\mathcal {X}^2_{4}(u(t))+\mathcal {X}^2_{4}(v(t))+\mathcal {X}^2_{8}(\Phi(t))\big)\big( {E}_{7}(u(t)) +{E}_{7}(v(t))\big).
\end{align}
\par
Combing \eqref{rule5}, \eqref{rtt56666666}, \eqref{y78900}, \eqref{4890uoop} and \eqref{ccfr5666699}, we can get
\begin{align}
&{E}_{7}(u(t)) +{E}_{7}(v(t))+\int_{0}^{t} \mathcal{E}_{7}(u(t)) +\mathcal{E}_{7}(v(t)) dt\nonumber\\
&\leq C\varepsilon^2+C\int_0^{t}\langle t\rangle^{-1} \big(\mathcal {X}^2_{4}(u(t))+\mathcal {X}^2_{4}(v(t))\big)\big(\mathcal {E}^{\frac{1}{2}}_{7}(u(t)) +\mathcal{E}^{\frac{1}{2}}_{7}(v(t))\big)\big( {E}^{\frac{1}{2}}_{7}(u(t)) +{E}^{\frac{1}{2}}_{7}(v(t))\big)dt\nonumber\\
&+ C\int_0^{t}\langle t\rangle^{-2}\big(\mathcal {X}^2_{4}(u(t))+\mathcal {X}^2_{4}(v(t))\big)\big( {E}_{7}(u(t)) +{E}_{7}(v(t))\big)dt\nonumber\\
&+ C\int_0^{t}\langle t\rangle^{-1} \mathcal {X}^2_{8}(\Phi(t))\big(\mathcal {E}^{\frac{1}{2}}_{7}(u(t)) +\mathcal{E}^{\frac{1}{2}}_{7}(v(t))\big)\big( {E}^{\frac{1}{2}}_{7}(u(t)) +{E}^{\frac{1}{2}}_{7}(v(t))\big)dt\nonumber\\
&+ C\int_0^{t}\langle t\rangle^{-2}\mathcal {X}^2_{8}(\Phi(t))\big( {E}_{7}(u(t)) +{E}_{7}(v(t))\big)dt
\nonumber\\
&\leq C\varepsilon^2+16CA_1^2A_2^2\varepsilon^4+ C\int_0^{t}\langle t\rangle^{-2}\big( {E}_{7}(u(t)) +{E}_{7}(v(t))\big)dt\nonumber\\
&+\frac{1}{100}\int_{0}^{t} \mathcal{E}_{7}(u(t)) +\mathcal{E}_{7}(v(t)) dt.
\end{align}
Then we have
\begin{align}
{E}_{7}(u(t)) +{E}_{7}(v(t))\leq C\varepsilon^2+16CA_1^2A_2^2\varepsilon^4+ C\int_0^{t}\langle t\rangle^{-2}\big( {E}_{7}(u(t)) +{E}_{7}(v(t))\big)dt
\end{align}
By Gronwall's inequality, we get
\begin{align}\label{gji899ener}
{E}^{\frac{1}{2}}_{7}(u(t)) +{E}^{\frac{1}{2}}_{7}(v(t))\leq C_0\varepsilon+4C_0A_1A_2\varepsilon^2.
\end{align}
\subsection{$L^{\infty}$ estimates}
By Lemma \ref{Linfty}, we have
\begin{align}
&\mathcal {X}_{4}(u(t))+\mathcal {X}_{4}(v(t))\nonumber\\
&\leq C\varepsilon+C\int_{0}^{t}\|F(u+\Theta, D(u+\Theta), Dv, D^2(u+\Theta), D^2v)\|_{\Gamma, 5,1}dt\nonumber\\
&~~~~~~~~+C\int_{0}^{t}\|G(u+\Theta, D(u+\Theta), Dv, D^2(u+\Theta), D^2v)\|_{\Gamma, 5,1}dt.
\end{align}
In view of \eqref{333333000}--\eqref{PF1dddd1111}, we have
\begin{align}\label{righttyy}
&\|F(u+\Theta, D(u+\Theta), Dv, D^2(u+\Theta), D^2v)\|_{\Gamma, 5,1}\nonumber\\
&~~~~~~~~~~+\|G(u+\Theta, D(u+\Theta), Dv, D^2(u+\Theta), D^2v)\|_{\Gamma, 5,1}\nonumber\\
&\leq \|\sin(2(u+\Theta))Q(v,v)\|_{\Gamma, 5,1}+ \|\sin(2(u+\Theta))Q(u+\Theta,v)\|_{\Gamma, 5,1}\nonumber\\
&+\|\sin^2(u+\Theta) \Box v\|_{\Gamma, 5,1}+\|\sin(2(u+\Theta))Q_{\mu\nu}(u+\Theta,v)Q^{\mu\nu}(u+\Theta,v)\|_{\Gamma, 5,1}\nonumber\\
&+\|\cos^2(u+\Theta) Q_{\mu\nu}
\big(v,Q^{\mu\nu}(u+\Theta,v)\big)\|_{\Gamma, 5,1}\nonumber\\
&+\|\cos^2(u+\Theta) Q_{\mu\nu}
\big(u+\Theta,Q^{\mu\nu}(u+\Theta,v)\big)\|_{\Gamma, 5,1}.
\end{align}
We will focus on the first three terms on the right hand side of \eqref{righttyy}, the remaining terms can be treated similarly. \par
For the first term on the right hand side of \eqref{righttyy},
it follows from Lemma \ref{composite} and Lemma \ref{QL} that
\begin{align}\label{5677yyy7}
&\|\sin(2(u+\Theta))Q(v,v)\|_{\Gamma,5,1}\nonumber\\
&\leq C\sum_{|b|+|\beta|\leq 5}\|\Gamma^{b}\sin(2(u+\Theta))\Gamma^{\beta}Q(v,v)\|_{L^1(\mathbb{R}^2)}\nonumber\\
&\leq C\langle t\rangle^{-1}\sum_{|b|+|c|+|d|\leq 5}\|\Gamma^{b}u D\Gamma^{c}v \Gamma^{d+1}v\|_{L^1(\mathbb{R}^2)}+ C\langle t\rangle^{-1}\sum_{|b|+|c|+|d|\leq 6}\|\Gamma^{b}\Theta D\Gamma^{c}v \Gamma^{d+1}v\|_{L^1(\mathbb{R}^2)}.
\end{align}
For $|b|+|c|+|d|\leq 5$,
if $|b|+|d|\leq 3$, we have
\begin{align}\label{xiaoyyypjk}
&\|\Gamma^{b}u D\Gamma^{c}v \Gamma^{d+1}v\|_{L^1(\mathbb{R}^2)}\nonumber\\
&\leq C\langle t\rangle^{-1} \|D\Gamma^{c}v\|_{L^{2}(\mathbb{R}^2)} \|\langle t-r\rangle^{-\frac{1}{2}}\langle t+r\rangle^{\frac{1}{2}}\langle t-r\rangle^{\frac{1}{2}}\Gamma^{b}u\|_{L^{4}(\mathbb{R}^2)}\nonumber\\
  &~~~\cdot\|\langle t-r\rangle^{-\frac{1}{2}}\langle t+r\rangle^{\frac{1}{2}}\langle t-r\rangle^{\frac{1}{2}}\Gamma^{d+1}v\|_{L^{4}(\mathbb{R}^2)}\nonumber\\
&\leq C\langle t\rangle^{-1} \|\langle t-r\rangle^{-\frac{1}{2}}\|^2_{L^{4}(|x|\leq t+1)}\mathcal {X}_{4}(u(t))\mathcal {X}_{4}(v(t)) {E}^{\frac{1}{2}}_{7}(v(t))\nonumber\\
&\leq C\langle t\rangle^{-\frac{1}{2}} \mathcal {X}_{4}(u(t))\mathcal {X}_{4}(v(t)) {E}^{\frac{1}{2}}_{7}(v(t))
\end{align}
If $|b|+|c|\leq 3$, by Hardy inequality \eqref{Hardy} and \eqref{gooddecay345}, we have
\begin{align}\label{fgrtyyyjjt}
&\|\Gamma^{b}u  D\Gamma^{c}v \Gamma^{d+1}v\|_{L^1(\mathbb{R}^2)}\nonumber\\
&\leq C\langle t\rangle^{-1} \|\langle t-r\rangle^{-1}\Gamma^{d+1}v\|_{L^2(\mathbb{R}^2)}\|\langle t-r\rangle^{-\frac{1}{2}}\langle t+r\rangle^{\frac{1}{2}}\langle t-r\rangle^{\frac{1}{2}}\Gamma^{b}u\|_{L^{4}(\mathbb{R}^2)}\nonumber\\
  &~~~\cdot\|\langle t-r\rangle^{-\frac{1}{2}}\langle t+r\rangle^{\frac{1}{2}}\langle t-r\rangle^{\frac{1}{2}}\Gamma^{c+1}v\|_{L^{4}(\mathbb{R}^2)}\nonumber\\
&\leq C\langle t\rangle^{-1} \|\langle t-r\rangle^{-\frac{1}{2}}\|^2_{L^{4}(|x|\leq t+1)}\mathcal {X}_{4}(u(t))\mathcal {X}_{4}(v(t)) {E}^{\frac{1}{2}}_{7}(v(t))\nonumber\\
&\leq C\langle t\rangle^{-\frac{1}{2}} \mathcal {X}_{4}(u(t))\mathcal {X}_{4}(v(t)) {E}^{\frac{1}{2}}_{7}(v(t)).
\end{align}
Similarly to \eqref{fgrtyyyjjt}, if $|c|+|d|\leq 3$, it holds that
\begin{align}\label{xyaohi89}
\|\Gamma^{b}u D\Gamma^{c}v \Gamma^{d+1}v\|_{L^1(\mathbb{R}^2)}\leq C\langle t\rangle^{-\frac{1}{2}} \mathcal {X}_{4}(v(t))\mathcal {X}_{4}(v(t)) {E}^{\frac{1}{2}}_{7}(u(t)).
\end{align}
Thus we obtain
\begin{align}\label{ui8900}
&\sum_{|b|+|c|+|d|\leq 5}\|\Gamma^{b}u D\Gamma^{c}v \Gamma^{d+1}v\|_{L^1(\mathbb{R}^2)}\nonumber\\
&\leq C\langle t\rangle^{-\frac{1}{2}}\big( \mathcal {X}^2_{4}(u(t))+\mathcal {X}^2_{4}(v(t))\big) \big({E}^{\frac{1}{2}}_{7}(u(t))+{E}^{\frac{1}{2}}_{7}(v(t))\big).
\end{align}
For the second part on the right hand side of \eqref{5677yyy7}, for $|b|+|c|+|d|\leq 5$,
if $|b|+|d|\leq 3$, similarly to \eqref{xiaoyyypjk}, we get
\begin{align}\label{xiaoyyypjrttttk}
\|\Gamma^{b}\Theta D\Gamma^{c}v \Gamma^{d+1}v\|_{L^1(\mathbb{R}^2)}
\leq C\langle t\rangle^{-\frac{1}{2}} \mathcal {X}_{4}(\Phi(t))\mathcal {X}_{4}(v(t)) {E}^{\frac{1}{2}}_{7}(v(t)).
\end{align}
If $|b|+|c|\leq 3$ or $|c|+|d|\leq 3$, similarly to \eqref{fgrtyyyjjt}, we have
\begin{align}\label{fgrtyyyoojjt}
\|\Gamma^{b}\Theta D\Gamma^{c}v \Gamma^{d+1}v\|_{L^1(\mathbb{R}^2)}
\leq C\langle t\rangle^{-\frac{1}{2}} \mathcal {X}_{8}(\Phi(t))\mathcal {X}_{4}(v(t)) {E}^{\frac{1}{2}}_{7}(v(t)).
\end{align}
Thus we obtain
\begin{align}\label{xyu7999}
\sum_{|b|+|c|+|d|\leq 6}\|\Gamma^{b}\Theta D\Gamma^{c}v \Gamma^{d+1}v\|_{L^1(\mathbb{R}^2)}\leq C\langle t\rangle^{-\frac{1}{2}} \big(\mathcal {X}^2_{8}(\Phi(t))+\mathcal {X}^2_{4}(v(t))\big) {E}^{\frac{1}{2}}_{7}(v(t)).
\end{align}
It follows from \eqref{5677yyy7}, \eqref{ui8900} and \eqref{xyu7999} that
\begin{align}\label{hj7899900}
&\|\sin(2(u+\Theta))Q(v,v)\|_{\Gamma,5,1}\nonumber\\
&\leq C\langle t\rangle^{-\frac{3}{2}}\big( \mathcal {X}^2_{4}(u(t))+\mathcal {X}^2_{4}(v(t))+\mathcal {X}^2_{8}(\Phi(t))\big) \big({E}^{\frac{1}{2}}_{7}(u(t))+{E}^{\frac{1}{2}}_{7}(v(t))\big).
\end{align}
\par
Similarly to \eqref{hj7899900}, the second term on the right hand side of \eqref{righttyy} can be estimated by the same by and admits the same upper bound.\par
For the third term on the right hand side of \eqref{righttyy}, by Lemma \ref{composite} and Lemma \ref{uu679yui}, we get
\begin{align}
&\|\sin^2(u+\Theta) \Box v\|_{\Gamma, 5,1}\nonumber\\
&\leq C\sum_{|b|+|\beta|\leq 5}\|\Gamma^{\beta}\sin^2(2(u+\Theta))\Box\Gamma^{b}v\|_{L^1(\mathbb{R}^2)}\nonumber\\
&\leq C\langle t\rangle^{-1}\sum_{|b|+|c|+|d|\leq 6}\|\Gamma^{c}u\Gamma^{d}uD\Gamma^{b}v\|_{L^1(\mathbb{R}^2)}+C\langle t\rangle^{-1}\sum_{|b|+|c|+|d|\leq 6}\|\Gamma^{c}u\Gamma^{d}\Theta D\Gamma^{b}v\|_{L^1(\mathbb{R}^2)}\nonumber\\
&+C\langle t\rangle^{-1}\sum_{|b|+|c|+|d|\leq 6}\|\Gamma^{c}\Theta\Gamma^{d}\Theta D\Gamma^{b}v\|_{L^1(\mathbb{R}^2)}.
\end{align}
Then similarly to \eqref{hj7899900}, we have
\begin{align}\label{hj789fff4r449900}
&\|\sin^2(u+\Theta) \Box v\|_{\Gamma, 5,1}\nonumber\\
&\leq C\langle t\rangle^{-\frac{3}{2}}\big( \mathcal {X}^2_{4}(u(t))+\mathcal {X}^2_{4}(v(t))+\mathcal {X}^2_{8}(\Phi(t))\big) \big({E}^{\frac{1}{2}}_{7}(u(t))+{E}^{\frac{1}{2}}_{7}(v(t))\big).
\end{align}
\par
From the above discussion, we obtain
\begin{align}\label{dfttyyy56}
&\mathcal {X}_{4}(u(t))+\mathcal {X}_{4}(v(t))\nonumber\\
&\leq C\varepsilon+C\int_{0}^{t}\langle t\rangle^{-\frac{3}{2}}\big( \mathcal {X}^2_{4}(u(t))+\mathcal {X}^2_{4}(v(t))+\mathcal {X}^2_{8}(\Phi(t))\big) \big({E}^{\frac{1}{2}}_{7}(u(t))+{E}^{\frac{1}{2}}_{7}(v(t))\big) dt\nonumber\\
&\leq C_1\varepsilon+2C_1A_1\varepsilon+8C_1A_1A_2^2\varepsilon^3.
\end{align}
\subsection{Conclusion of the proof}
Noting \eqref{gji899ener} and \eqref{dfttyyy56}, we get
\begin{align}
\sup_{0\leq t\leq T}\big(E_{7}^{\frac{1}{2}}(u(t))+E_{7}^{\frac{1}{2}}(v(t))\big)\leq C_0\varepsilon+4C_0A_1A_2\varepsilon^2
\end{align}
and
\begin{align}
\sup_{0\leq t\leq T}\big(\mathcal {X}_{4}(u(t))+\mathcal {X}_{4}(v(t))\big)
\leq C_1\varepsilon+2C_1A_1\varepsilon+8C_1A_1A_2^2\varepsilon^3.
\end{align}
Assume that
\begin{align}
E_{7}^{\frac{1}{2}}(u(0))+E_{7}^{\frac{1}{2}}(v(0))\leq \widetilde{C}_0\varepsilon~~\text{and}~~
\mathcal {X}_{4}(u(0))+\mathcal {X}_{4}(v(0))\leq \widetilde{C}_1\varepsilon.
\end{align}
Take $A_1=\max\{4C_0,4 \widetilde{C}_0\}$, $A_2=\max\{8(C_1+2C_1A_1), 4 \widetilde{C}_1\}$ and $\varepsilon_0$ sufficiently small such that
\begin{align}
16C_0A_2\varepsilon_0+32C_1A_1A_2\varepsilon_0^2\leq 1.
\end{align}
Then for any $0<\varepsilon\leq \varepsilon_0$, we have
\begin{align}
\sup_{0\leq t\leq T}\big(E_{7}^{\frac{1}{2}}(u(t))+E_{7}^{\frac{1}{2}}(v(t))\big)\leq A_1\varepsilon~~\text{and}~~\sup_{0\leq t\leq T}\big(\mathcal {X}_{4}(u(t))+\mathcal {X}_{4}(v(t))\big)\leq A_2\varepsilon,
\end{align}
which completes the proof of Theorem \ref{mainthm2}.

\end{document}